\documentclass[10pt]{amsart}
\usepackage{amssymb, amstext, amscd, amsmath, amssymb}
\usepackage{mathtools, color, paralist, dsfont, rotating}
\usepackage{mathrsfs} 
\usepackage{verbatim}
\usepackage{enumerate}
\usepackage{amsrefs}
\usepackage[colorlinks, citecolor=blue, linkcolor=red, pdfstartview=FitB, bookmarks]{hyperref}
\usepackage[latin1]{inputenc}

\usepackage{tikz-cd}
\usepackage{tikz}

\numberwithin{equation}{section}

\makeatletter
\def\@cite#1#2{{\m@th\upshape\bfseries%
[{#1\if@tempswa{\m@th\upshape\mdseries, #2}\fi}]}}
\makeatother
\theoremstyle{plain}
\newtheorem{theorem}{Theorem}[section]
\newtheorem{corollary}[theorem]{Corollary}
\newtheorem{proposition}[theorem]{Proposition}
\newtheorem{lemma}[theorem]{Lemma}
\newtheorem{hypothesis}{Hypothesis}
\theoremstyle{definition}
\newtheorem{definition}[theorem]{Definition}
\newtheorem{example}[theorem]{Example}

\newtheorem{question}{Question}

\theoremstyle{remark}

\mathtoolsset{centercolon}
  
\renewcommand{\phi}{\varphi}  
\renewcommand{\epsilon}{\varepsilon}  
\newcommand{\eps}{{\varepsilon}}
\newcommand{\A}{{\mathcal{A}}}
\newcommand{\B}{{\mathcal{B}}}

\newcommand{\F}{{\mathcal{F}}}

\renewcommand{\H}{{\mathcal{H}}}

\newcommand{\J}{{\mathcal{J}}}

\renewcommand{\L}{{\mathcal{L}}}

\newcommand{\Q}{{\mathcal{Q}}}
\newcommand{\R}{{\mathcal{R}}}
\renewcommand{\S}{{\mathcal{S}}}

\newcommand{\X}{{\mathcal{X}}}
\newcommand{\Y}{{\mathcal{Y}}}

\newcommand{\rA}{{\mathrm{A}}}
\newcommand{\rC}{{\mathrm{C}}}

\newcommand{\bB}{\mathbb{B}}
\newcommand{\bC}{\mathbb{C}}
\newcommand{\bD}{\mathbb{D}}

\newcommand{\bN}{\mathbb{N}}
\newcommand{\bM}{\mathbb{M}}
\newcommand{\bT}{\mathbb{T}}
\newcommand{\bZ}{\mathbb{Z}}

\renewcommand{\bZ}{\mathbb{Z}}

\newcommand{\fA}{{\mathfrak{A}}}
\newcommand{\fB}{{\mathfrak{B}}}
\newcommand{\fC}{{\mathfrak{C}}}

\newcommand{\fD}{{\mathfrak{D}}}

\newcommand{\fH}{{\mathfrak{H}}}

\newcommand{\fK}{{\mathfrak{K}}}

\newcommand{\fM}{{\mathfrak{M}}}
\newcommand{\fN}{{\mathfrak{N}}}

\newcommand{\fS}{{\mathfrak{S}}}

\newcommand{\fV}{{\mathfrak{V}}}

\newcommand{\qand}{\quad\text{and}\quad}

\newcommand{\ol}{\overline}

\newcommand{\id}{{\operatorname{id}}}

\newcommand{\Mult}{{\operatorname{Mult}}}

\newcommand{\ca}{\mathrm{C}^*}

\begin{document}

\title[Finite-dimensional approximations and semigroup coactions]{Finite-dimensional approximations and semigroup coactions for operator algebras}

\author[R. Clou\^atre]{Rapha\"el Clou\^atre}
\address{Department of Mathematics \\ University of Manitoba \\ Winnipeg \\ MB \\ Canada}
\email{Raphael.Clouatre@umanitoba.ca}

\author[A. Dor-On]{Adam Dor-On}
\address{Department of Mathematics\\ University of Copenhagen\\ Copenhagen\\ Denmark}
\email{adoron@math.ku.dk}


\thanks{The first author was partially supported by an NSERC Discovery grant. The second author was partially supported by NSF grant DMS-1900916 and by the European Union's Horizon 2020 Marie Sk\l{}odowska-Curie grant No 839412.}

\begin{abstract}
The residual finite-dimensionality of a $\rC^*$-algebra is known to be encoded in a topological property of its space of representations, stating that finite-dimensional representations should be dense therein. We extend this paradigm to general (possibly non-selfadjoint) operator algebras. While numerous subtleties emerge in this greater generality, we exhibit novel tools for constructing finite-dimensional approximations. One such tool is a notion of a residually finite-dimensional coaction of a semigroup on an operator algebra, which allows us to construct finite-dimensional approximations for operator algebras of functions and operator algebras of semigroups. Our investigation is intimately related to the question of when residual finite-dimensionality of an operator algebra is inherited by its maximal $\rC^*$-cover, which we establish in many cases of interest.
\end{abstract}

\maketitle

\section{Introduction}

In the classification program for certain classes of concrete non-selfadjoint operator algebras, a common theme is the analysis of their finite-dimensional representations; see for instance \cites{DRS2011, hartz2017isom, doron2018, SSS2018} and the references therein. Although this is a perfectly natural and effective approach in these situations, in general it is not entirely clear to what extent finite-dimensional representations truly capture the structure of a given operator algebra. Indeed, one has to wonder whether, and to what extent, finite-dimensional representations are sufficient to determine the behavior of arbitrary representations by approximation. This is the case for operator algebras that are residually finite-dimensional.

An operator algebra is said to be \emph{residually finite-dimensional} (or RFD) if it admits a completely norming set of representations on finite-dimensional Hilbert spaces. For $\rC^*$-algebras, this is a classical notion that has been thoroughly investigated. Yet, the general notion has received comparatively little attention beyond \cite{mittal2010} and \cite{AHMR2020subh}, both of which focus on concrete algebras of functions. In \cite{CR2018rfd} a systematic study of RFD operator algebras was undertaken, where several operator algebra were shown to be RFD. Even for concrete classes of operator algebras it is challenging to characterize which of them are RFD, and this provides the impetus for developing alternative characterizations of this notion. However, at present there are no such tools available for general operator algebras. This stands in sharp contrast with the more classical $\rC^*$-algebraic setting where various conditions have been shown to be equivalent to residual finite-dimensionality \cites{EL1992, archbold1995, hadwin2014, CS2019}. In this paper, we provide new characterizations that apply to general operator algebras.

Our approach is inspired by the work of Exel and Loring \cite{EL1992}, which we discuss next. Let $\fA$ be a $\rC^*$-algebra, and recall that the spectrum of $\fA$ is the set $\widehat \fA$ of unitary equivalence classes of irreducible $*$-representations, together with the point strong operator topology. In \cite{EL1992} it is shown that $\fA$ is RFD if and only if for every $*$-representation there exists a net $(\pi_\lambda)$ of finite dimensional $*$-representations of $\fA$ that converges to $\pi$ in the point strong operator topology. Following \cite{EL1992}, it was shown by Archbold \cite{archbold1995} that a $\rC^*$-algebra $\fA$ is RFD if and only if the set of unitary equivalence classes of finite-dimensional irreducible $*$-representations is dense in $\widehat\fA$. Hence, the two topological characterizations of residual finite-dimensionality, one in terms of irreducible $*$-representations, and another in terms of arbitrary $*$-representations, provides a fair amount of additional flexibility.

The main driving force behind the present paper is the investigation of those representations of general operator algebras that admit a finite-dimensional approximation of the type described in the previous paragraph. As we will see, there are subtleties intrinsic to working with general operator algebras (as opposed to $\rC^*$-algebras), wherein there are several different reasonable interpretations of what a finite-dimensional approximation should be. Interestingly, this flexibility allows one to connect the existence of finite-dimensional approximations to the residual finite-dimensionality of various $\rC^*$-covers for the operator algebra. This is eminently desirable, as RFD $\rC^*$-algebras are typically much better understood. We also note that the properties of the maximal and minimal $\rC^*$-covers being RFD have been explored previously in \cites{CR2018rfd, CTh2022}, and the findings of these papers aptly illustrate the depth of the problem. Therefore, our analysis based on the existence of finite-dimensional approximations sheds light on some unresolved problems.

Next, we describe the organization of the paper. In Section \ref{S:fdimapprox}, we introduce various notions of finite-dimensional approximations, which we briefly mention here. In what follows, all representations of operator algebras are completely contractive. Let $\A$ be an operator algebra and let $\pi:\A\to B(\fH)$ be a representation. Let
\[
\pi_\lambda:\A\to B(\fH), \quad \lambda\in \Lambda
\]
be a net of representations with the property that the space $\rC^*(\pi_\lambda(\A))\fH$ is finite-dimensional for every $\lambda\in \Lambda$. We say that the net $(\pi_\lambda)$ is a \emph{finite-dimensional approximation} for $\pi$ if $(\pi_\lambda(a))$ converges to $\pi(a)$ in the SOT for every $a\in \A$. If, in addition, we have that $(\pi_\lambda(a)^*)$ converges to $\pi(a)^*$ in the SOT for every $a\in \A$, then we say that $(\pi_\lambda)$ is an \emph{Exel--Loring approximation} for $\pi$. These two notions coincide when $\A$ is a $\rC^*$-algebra. Finite-dimensional approximations do not necessarily arise in geometrically transparent ways (see Example \ref{E:seminv}) and thus, constructing such approximations can be highly non-trivial. Consequently, it is beneficial to identify operations that preserve the existence of these approximations, and this is the focus of the rest of Section \ref{S:fdimapprox}.

In Section \ref{S:charRFD}, we utilize finite-dimensional approximations to analyze the residual finite-dimensionality of general operator algebras. Recall that the \emph{$\rC^*$-envelope} of $\A$ is the smallest $\rC^*$-algebra generated by a copy $\A$ and it is denoted by $\rC^*_e(\A)$. As explained in Subsection \ref{SS:C*cov}, it can be constructed with the aid of so-called \emph{extremal} representations of $\A$. Using these ideas, we can now state one of our main results, which is a version of \cite[Theorem 2.4]{EL1992} that characterizes residual finite-dimensionality for general operator algebras (see Theorems \ref{T:ELOA} and \ref{T:RFDcharC*env}).

\begin{theorem}\label{T:mainA}
Let $\A$ be a separable operator algebra. Then, the following statements are equivalent.
\begin{enumerate}[{\rm (i)}]
\item The algebra $\A$ is RFD.
\item  Every extremal representation of $\A$ has a finite-dimensional approximation.
\item  Every extremal representation of $\A$ has an Exel--Loring approximation.
\item If $\pi$ is a $*$-representation of $\rC^*_e(\A)$, then $\pi|_\A$ admits an Exel--Loring approximation.
\end{enumerate}
\end{theorem}
The curious reader may wonder how the statements appearing above relate to having an RFD $\rC^*$-envelope. It is easy to construct examples where $\A$ is RFD while $\rC^*_e(\A)$ is not, but a precise condition on $\A$ which is equivalent to $\rC^*_e(\A)$ being RFD is unknown; see \cites{CTh2022} and \cite[Proposition 4.1]{AHMR2020subh} for recent related results.

At the other extreme in the scale of $\rC^*$-algebras generated by a copy of $\A$ is the \emph{maximal} $\rC^*$-cover, denoted by $\rC^*_{\max}(\A)$. Notably, the existence of Exel--Loring approximations can be used to characterize the residual finite-dimensionality of $\rC^*_{\max}(\A)$, as shown in Theorem \ref{T:ELC*max}.

\begin{theorem}\label{T:mainB}
 Let $\A$ be an operator algebra. Then $\rC^*_{\max}(\A)$ is RFD if and only if every representation of $\A$ admits an Exel--Loring approximation.
 \end{theorem}
 
Motivated by these results, let us say that the operator algebra $\A$ has property $(\mathscr{F})$ if all of its representations admit a finite-dimensional approximation. The question of which operator algebras enjoy property $(\mathscr{F})$ is an interesting one, and Theorems \ref{T:mainA} and \ref{T:mainB} imply that
\[
 \text{$\rC^*_{\max}(\A)$ is RFD} \quad  \Rightarrow  \quad \text{$\A$ has property $(\mathscr{F})$}  \quad \Rightarrow \quad 
\text{$\A$ is RFD}.
\]
Equivalently, we are ultimately searching for situations when a non-self-adjoint version of the theorem of Exel and Loring holds with finite-dimensional approximations, or Exel-Loring approximations. It is important to mention that after the appearance of our paper on arXiv, Hartz was able to find a complete characterization of RFD operator algebras in terms of finite-dimensional approximations \cite[Theorem 1.2]{Har+}. More precisely, he showed that an operator algebra $\A$ is RFD if and only if for every representation $\pi : \A \rightarrow B(\mathfrak{H})$ there is a net $(\pi_{\lambda})$ of finite-dimensional representations of $\A$ on $\mathfrak{H}$ such that $(\pi_{\lambda}(a))$ converges to $\pi(a)$ in the WOT (rather than the SOT).

In trying to understand property $(\mathscr{F})$, a natural question emerges.

\begin{question}\label{Q:rfdC*max}
Let $\A$ be an RFD operator algebra. When is $\rC^*_{\max}(\A)$ RFD?
\end{question}

This question was already raised in \cite[Section 5]{CR2018rfd} with some partial results; see also \cite{thompson2022} for recent progress. Our results show that Question \ref{Q:rfdC*max} is equivalent to asking when is it the case that the \emph{existence} of a representation admitting a finite-dimensional approximation implies that \emph{all} representations admit an Exel--Loring approximation. Our approximation techniques and the resulting property $(\mathscr{F})$ thus unearth an intermediate state worthy of investigation.

\begin{question}\label{Q:propF}
Let $\A$ be an RFD operator algebra. When does $\A$ have property $(\mathscr{F})$? In other words, when do all representations of $\A$ admit a finite-dimensional approximation?
\end{question}

After this paper appeared on arXiv, Hartz was able to find an example of an operator algebra $\A$ such that $\rC^*_{\max}(\A)$ is not RFD while $\A$ is RFD \cite[Theorem 1.3]{Har+}. In fact, this  operator algebra $\A$ also fails to have property $(\mathscr{F})$. Thus, we see that there are representations of this RFD operator algebra $\A$ which do not even admit a finite-dimensional approximation, where pointwise convergence occurs in the SOT. This means that without restricting the class of representations under consideration (as is done in Theorem \ref{T:mainA} for instance), \cite[Theorem 1.2]{Har+} is sharp, and constitutes a genuine generalization of the Exel--Loring theorem to the non-self-adjoint context.  In light of this, our Theorem \ref{T:mainB} is particularly enlightening, as it offers an optimal characterization of residual finite-dimensionality of $\rC^*_{\max}(\A)$ in terms of \emph{all} representations of $\A$.

The rest of our paper is motivated by Question \ref{Q:rfdC*max}, in that we find concrete operator algebras whose maximal $\rC^*$-cover is RFD, and thus do not display the pathological behaviour of the example from \cite{Har+}.

While $\rC^*$-algebras are analogous to groups, general operator algebras are analogous to semigroups. Hence, non-self-adjoint operator algebras arising from semigroups form a natural class of examples for our study. Semigroup $\rC^*$-algebras have been studied by many authors over the years \cite{nica1992, LR1996, CL2002, CL2007, CEL2015, BLS2017, aHRT2018, aHNSY2021}, where a unified approach was obtained by Li for the class of \emph{independent} semigroups in \cite{li2012, li2013, li2017}. An operator algebra analogue of the reduced group $\rC^*$-algebra is $\A_r(P)$, which is the operator algebra generated by the left regular representation of $P$ on $\ell^2(P)$. 

Suppose $P$ is a cancellative semigroup (both left and right). For $p\in P$ we define its set of right divisors 
\[
\R_p=\{r\in P:p=qr \text{ for some } q\in P\}.
\]
 We say that $P$ has the \emph{finite divisor property} (or FDP) if $\R_p$ is finite for every $p\in P$. It is relatively straightforward to show that $\A_r(P)$ is RFD when $P$ has FDP (see Proposition \ref{P:ArP}). With the aim of answering Question \ref{Q:rfdC*max} in concrete examples, we introduce in Section \ref{S:coactions} a notion of an RFD semigroup coaction of $P$ on an operator algebra $\A$, and prove the following (see Theorem \ref{T:coactionC*maxRFD}).

\begin{theorem} \label{T:mainE}
Let $P$ a countable cancellative semigroup with FDP such that there is a character $\chi:\A_r(P)\to \bC$ with $\chi(\lambda_p)=1$ for every $p\in P$. Let $\A$ be a separable operator algebra that admits an RFD coaction of $P$. Then, $\rC^*_{\max}(\A)$ is RFD.
\end{theorem}

Coactions of groups on general operator algebras were introduced in \cite{DKKLL+} for the purpose of showing the existence of $\rC^*$-algebras satisfying co-universality with respect to representations of product systems over right LCM monoids. In the $\rC^*$-algebra literature, coactions by a discrete group $G$ are interpreted as actions of the quantum dual group of $G$ on the $\rC^*$-algebra, and are useful in many instances \cite{Quigg1996, Exe1997, CLSV2011, Seh19}. Since our semigroups are assumed to be FDP, they will not contain infinite groups as subsemigroups, and hence our development of coactions by semigroups differs from the one for group coactions in \cite{DKKLL+}.

The class of independent semigroups includes the majority of examples of semigroups previously studied through their $\rC^*$-algebras in the literature. In Section \ref{SS:semigroup} we use Theorem \ref{T:mainE} to show that the maximal C*-cover is RFD for semigroup algebras of many independent semigroups, as well as for operator algebras of functions with circular symmetry (see Theorem \ref{T:main-semigroup} and Theorem \ref{T:coactionAsymm} respectively). 

\begin{theorem}\label{T:mainF}
Let $P$ and $Q$ be countable semigroups embedded in some groups $G$ and $H$ respectively. Assume that $P$ and $Q$ are independent, and that $Q$ is left-amenable and has FDP. Assume further that there is a homomorphism $\varphi : G \rightarrow H$ such that $\varphi(P) \subset Q$ and $\varphi|_{P}$ is finite-to-one. Then, there is an RFD coaction of $Q$ on $\A_r(P)$ and $\rC^*_{\max}(\A_r(P))$ is RFD.
\end{theorem}

This theorem allows us to show that for many concrete examples of independent semigroups, the operator algebra $\A_r(P)$ has an RFD maximal $\rC^*$-cover. For instance, this works for $\mathbb{N}^d$ for any integer $d$, left-angled Artin monoids (Example \ref{ex:laam}) and Braid monoids (Example \ref{ex:braid}). 

\begin{theorem}\label{T:mainC}
Let $\Omega\subset\bC^n$ be a balanced subset and let $\A$ be a $\bT$-symmetric operator algebra of functions on $\Omega$. Assume that there is a collection of homogeneous polynomials spanning a dense subset of $\A$. Then, $\A$ admits an RFD coaction by $\bN$. In particular, $\rC^*_{\max}(\A)$ is RFD.
\end{theorem}

This theorem is applied to some uniform algebras (Corollary \ref{C:AOmega}), to the famous Schur--Agler class of functions (Corollary \ref{C:SAclass}) and to some algebras of multipliers on well-behaved reproducing kernel Hilbert spaces (Corollary \ref{C:RKHS}). Taking into account Theorem \ref{T:mainB}, these results shed considerable light on the structure of representations of operator algebras that play an important role in multivariate operator theory; see \cite{CD2016abscont, CH2018, BHM2018, CT2021} and the references therein. 

In particular, it follows from Corollary \ref{C:AOmega} that the maximal $\rC^*$-cover of the polydisc algebra $\rA(\bD^n)$ is RFD. This result is particularly striking in view of the following. First, it is a consequence of Ando's theorem \cite[Theorem 5.5]{paulsen2002} that $\rC^*_{\max}(\rA(\bD^2))$ is the universal $\rC^*$-algebra generated by a pair of commuting contractions. Hence, this universal C*-algebra is RFD. Second, we know that $\rA(\bD^2)$ is a uniform algebra, and hence it has a commutative $\rC^*$-envelope. Thus, $\rA(\bD^2)$ is an RFD operator algebra with RFD minimal and maximal $\rC^*$-covers.

Nevertheless, there exists an intermediate $\rC^*$-cover which is \emph{not} RFD (such a phenomenon had already been observed in \cite[Example 6]{CR2018rfd}). Indeed, recall that two commuting contractions $T_1,T_2$ are said to be \emph{doubly commuting} if $T_1^*T_2=T_2T_1^*$. Then, it follows from \cite[Theorem 6.11]{CShe2020} and the refutation of Connes embedding conjecture \cite{JNVWY2020} that the universal $\rC^*$-algebra $\fD$ generated by a pair of \emph{doubly} commuting contractions is not RFD. Since the standard representation of $\rA(\bD^2)$ on the Hardy space of the bidisc is generated by a pair doubly commuting isometries, Ando's inequality implies that $\fD$ is a $\rC^*$-cover of $\rA(\bD^2)$.


\section{Finite-dimensional approximations}\label{S:fdimapprox}

In this section, we examine various notions of finite-dimensional approximations for representations of operator algebras. Before proceeding however, we need to recall some operator algebraic background.

\subsection{$\ca$-covers and extremal representations}\label{SS:C*cov}

Let $\A$ be an operator algebra. If $\fB$ is a $\rC^*$-algebra and $\iota:\A\to \fB$ is a completely isometric homomorphism with $\fB=\rC^*(\iota(\A))$, then we say that the pair $(\fB,\iota)$ is a \emph{$\rC^*$-cover} of $\A$. Throughout the paper, we will be dealing with two special $\rC^*$-covers, which we introduce next.

The \emph{maximal} $\rC^*$-cover of $\A$, denoted by $(\rC^*_{\max}(\A),\upsilon)$, is the essentially unique $\rC^*$-cover satisfying the following universal property: for any $\rC^*$-algebra $\fC$ and any representation $\pi:\A\to \fC$, there is a $*$-representation $\widehat\pi:\rC^*_{\max}(\A)\to \fC$ such that $\widehat\pi\circ\upsilon=\pi$. (By a \emph{representation} of an operator algebra, here and elsewhere we mean a completely contractive homomorphism.) For a proof of existence of the maximal $\rC^*$-cover, the reader should consult \cite[Proposition 2.4.2]{BLM2004}. 

The other $\rC^*$-cover that is relevant for our purposes is the ``minimal" one. It is usually referred to as the \emph{$\rC^*$-envelope of $\A$}, denoted by $(\rC^*_e(\A),\eps)$. Its defining universal property is the following: for any $\rC^*$-cover $(\iota,\fB)$ of $\A$, there is a $*$-representation $\pi:\fB\to \rC^*_e(\A)$ such that $\pi\circ \iota=\eps$. The existence of this object is highly non-trivial, see \cite[Theorem 4.3.1 and Proposition 4.3.5]{BLM2004}. 

Due to important work of Dritschel--McCullough \cite{dritschel2005} (inspired by previous insight of Muhly--Solel \cite{MS1998}), it is known that certain special representations of $\A$ are intimately related to its $\rC^*$-envelope. We recall these important concepts. 

Let $\pi:\A\to B(\fH_\pi)$ and $\theta:\A\to B(\fH_\theta)$ be representations. We say that $\theta$ is a \emph{dilation} of $\pi$ if $\fH_\pi\subset \fH_\theta$ and
\[
\pi(a)=P_{\fH_\pi}\theta(a)|_{\fH_\pi}, \quad a\in \A. 
\]
If in addition the space $\fH_\pi$ is invariant for $\theta(\A)$, then the dilation $\theta$ is said to be an \emph{extension} of $\pi$. Similarly, if $\fH_\pi$ is invariant for $\theta(\A)^*$, then the dilation $\theta$ is said to be a \emph{coextension} of $\pi$.

 If $\theta$ is a dilation of $\pi$, then it is said to be a \emph{trivial dilation} if the space $\fH_\pi$ is reducing for $\theta(\A)$. In other words, a trivial dilation $\theta$ of $\pi$ can be written as
\[
\theta(a)=\pi(a)\oplus \left(\theta(a)|_{\fH_\pi^\perp}\right), \quad a\in \A.
\]
A representation is said to be \emph{extension-extremal} if all its extensions are trivial. Likewise, it is said to be \emph{coextension-extremal} if all its coextensions are trivial. Moreover, a representation is simply said to be \emph{extremal} if all its dilations are trivial. Clearly, an extremal representation is necessarily extension-extremal and coextension-extremal. The reader should consult \cite{DK11} for a thorough study of these notions.
%

%

We now collect some minor variations on known facts that we require later. We provide a detailed argument, as some care must be taken when dealing with the unit (or lack thereof).

\begin{theorem}\label{T:extreps}
Let $(\fB,\iota)$ be a $\rC^*$-cover of some operator algebra $\A$. Let $\pi:\A\to B(\fH_\pi)$ be a representation. The following statements hold.
\begin{enumerate}[{\rm (i)}]

\item If $\pi$ is extremal, then there is a $*$-representation $\sigma:\fB\to B(\fH_\pi)$ such that $\pi=\sigma\circ \iota$.

\item If $\pi$ is extension-extremal, then there is an extremal representation $\sigma$ of $\A$ which is a coextension of $\pi$.

\item If $\pi$ is coextension-extremal, then there is an extremal representation $\sigma$ of $\A$ which is an extension of $\pi$.

%
\end{enumerate}
\end{theorem}
\begin{proof}
(i) If $\A$ is unital, then necessarily $\iota$ and $\fB$ are both unital and the desired statement is an immediate consequence of Stinespring's dilation theorem \cite[Theorem 4.1]{paulsen2002} applied to the representation $\pi\circ \iota^{-1}$.

If $\A$ is not unital, then \cite[Section 3]{meyer2001} implies the existence of a unital representation $\pi^1$ of the unitization $\A^1$ of $\A$ (see \cite[Paragraph 2.1.11]{BLM2004}) that extends $\pi$. It is easy to see that $\pi^1$ is still extremal (see for instance \cite[Proposition 2.5]{DS2018}). Hence, it suffices to apply the previous argument to $\pi^1$.

(ii) Assume first that $\A$ is unital. Then, $\pi(I)$ is a self-adjoint projection with range $\fM$ commuting with $\pi(\A)$. Thus, the map
$\pi':\A\to B(\fM)$ defined as 
\[
\pi'(a)=\pi(a)|_{\fM}, \quad a\in \A
\]
is a unital representation. It is readily verified that $\pi'$ is extension-extremal since $\pi$ is assumed to be. Invoking \cite[Corollary 3.11]{DK11}, we see that there is an extremal unital representation $\sigma':\A\to B(\fH_\sigma)$ which is a coextension of $\pi'$. Define now a representation $\sigma:\A\to B(\fH_\pi)$ as
\[
\sigma(a)=\sigma'(a)\oplus 0, \quad a\in \A.
\]
Clearly, this is a coextension of $\pi$.  To see that $\sigma$ is extremal, we argue as follows. First, we note that the zero map on $\iota(\A)$ has a unique completely positive extension to $\fB$, namely the zero map itself. Invoking \cite[Proposition 2.4]{DS2018}, we conclude that $0$ is an extremal representation of $\A$. Once we know this, we may argue as in the proof of \cite[Proposition 4.4]{arveson2011} (using the version of the Schwarz inequality found in \cite[Proposition 1.5.7]{BO2008}) to see that $\sigma'$ being extremal implies that $\sigma$ is also extremal, by another application of \cite[Proposition 2.4]{DS2018}.

If $\A$ is not unital, then as above we may apply \cite[Section 3]{meyer2001} to find a unital representation $\pi^1$ of the unitization $\A^1$ of $\A$. A trivial modification of the argument in \cite[Proposition 2.5]{DS2018} shows that $\pi^1$ is still extension-extremal. An application of the previous argument to $\pi^1$ completes the proof.

(iii) This is completely analogous to (ii).
\end{proof}

%
%
%

\subsection{Finite-dimensional approximations}\label{SS:fdimapprox}
 
Let $\A$ be an operator algebra and let $\pi:\A\to B(\fH)$ be a representation. Let
\[
\pi_\lambda:\A\to B(\fH), \quad \lambda\in \Lambda
\]
be a net of representations with the property that the space $\rC^*(\pi_\lambda(\A))\fH$ is finite-dimensional for every $\lambda\in \Lambda$. We say that the net $(\pi_\lambda)$ is a \emph{finite-dimensional approximation} for $\pi$ if $(\pi_\lambda(a))$ converges to $\pi(a)$ in the SOT for every $a\in \A$. If, instead,  we have that $(\pi_\lambda(a)^*)$ converges to $\pi(a)^*$ in the SOT for every $a\in \A$, then we say that $(\pi_\lambda)$ is a \emph{finite-dimensional $*$-approximation} for $\pi$. Finally, we say that $(\pi_\lambda)$ is an \emph{Exel--Loring approximation} for $\pi$ if it is both a finite-dimensional approximation and a finite-dimensional $*$-approximation. Our definitions above are directly inspired by work on Exel--Loring \cite[Theorem 2.4]{EL1992}, according to which all $*$-representations of an RFD $\rC^*$-algebra admit what we call Exel--Loring approximations.

The rest of this section is devoted to establishing basic facts about finite-dimen\-sional approximations that we require throughout the paper. First, we examine direct sums.

\begin{lemma}\label{L:directsum}
Let $\A$ be an operator algebra and let $\Omega$ be a set. For each $\omega\in\Omega$,  let $\pi_\omega:\A\to B(\fH_\omega)$ be a representation that admits an Exel--Loring approximation.  Let $\pi:\A\to B(\oplus_{\omega\in \Omega} \fH_\omega)$ be the representation defined as
\[
\pi(a)=\bigoplus_{\omega\in \Omega}\ \pi_\omega(a), \quad a\in \A.
\]
Then, $\pi$ admits an Exel--Loring approximation.
\end{lemma}
\begin{proof}
For each $\omega\in \Omega$, there is a directed set $\Lambda_\omega$ and a net 
\[
\pi_{\omega,\lambda}:\A\to B(\fH_\omega), \quad \lambda\in \Lambda_\omega
\]
which is an Exel--Loring approximation for $\pi_\omega$. 
Let $\F$ denote the directed set of finite subsets of $\Omega$. Given $F=\{\omega_1,\omega_2,\ldots,\omega_n\}\in \F$, it is readily verified that
\[
\lim_{\lambda_1\in \Lambda_{\omega_1}} \lim_{\lambda_2\in \Lambda_{\omega_2}}\cdots  \lim_{\lambda_n\in \Lambda_{\omega_n}}\left(\bigoplus_{i=1}^n \pi_{\omega_i,\lambda_i}(a) \oplus \bigoplus_{\omega\notin \F}0\right)=\bigoplus_{\omega\in F}\pi_\omega(a) \oplus \bigoplus_{\omega\notin \F}0
\]
and
\[
\lim_{\lambda_1\in \Lambda_{\omega_1}} \lim_{\lambda_2\in \Lambda_{\omega_2}}\cdots  \lim_{\lambda_n\in \Lambda_{\omega_n}}\left(\bigoplus_{i=1}^n \pi_{\omega_i,\lambda_i}(a)^* \oplus \bigoplus_{\omega\notin \F}0\right)=\bigoplus_{\omega\in F}\pi_\omega(a)^*\oplus \bigoplus_{\omega\notin \F}0
\]
in the SOT for every $a\in \A$. On the other hand, we see that
\[
\lim_{F\in \F}\left(\bigoplus_{\omega\in F}\pi_\omega(a) \oplus  \bigoplus_{\omega\notin \F}0\right)=\pi(a) \qand \lim_{F\in \F}\left(\bigoplus_{\omega\in F}\pi_\omega(a)^* \oplus  \bigoplus_{\omega\notin \F}0\right)=\pi(a)^*
\]
in the SOT for every $a\in \A$. Thus, there is an Exel--Loring approximation for $\pi$ contained in the set
\[
\left\{\bigoplus_{i=1}^n \pi_{\omega_i,\lambda_i} \oplus \bigoplus_{\omega\neq {\omega_1,\ldots,\omega_n}}0\right\}
\]
where $n\in \bN, \ \omega_1,\omega_2,\ldots,\omega_n\in\Omega, \ \lambda_1\in \Lambda_{\omega_1}, \lambda_2\in \Lambda_{\omega_2}, \ldots,\lambda_n\in \Lambda_{\omega_n}$.
\end{proof}

Two representations $\pi:\A\to B(\fH_\pi)$ and $\sigma:\A\to B(\fH_\sigma)$ are said to be \emph{approximately unitarily equivalent} if there is a sequence of unitary operators
\[
U_n:\fH_\pi\to \fH_\sigma, \quad n\in \bN
\]
with the property that
\[
\lim_{n\to\infty}\|U_n \pi(a) U_n^*-\sigma(a)\|=0, \quad a\in \A.
\]
Another elementary fact we need is that the existence of Exel--Loring approximations is preserved under approximate unitary equivalence. 

\begin{lemma}\label{L:ELaue}
Let $\A$ be an operator algebra and let $\pi:\A\to B(\fH_\pi)$ be a representation that admits an Exel--Loring approximation. Let $\sigma:\A\to B(\fH_\sigma)$ be another representation which is approximately unitarily equivalent to $\pi$. Then, $\sigma$ admits an Exel--Loring approximation.
\end{lemma}
\begin{proof}
By assumption, there is a sequence of unitary operators
\[
U_n:\fH_\pi\to \fH_\sigma, \quad n\in \bN
\]
with the property that
\begin{equation}\label{Eq:aue}
\lim_{n\to\infty}\| U_n \pi(a) U_n^*-\sigma(a)\|=0, \quad a\in \A.
\end{equation}
Obviously, we also have that
\begin{equation}\label{Eq:aue*}
\lim_{n\to\infty}\| U_n \pi(a) ^*U_n^*-\sigma(a)^*\|=0, \quad a\in \A.
\end{equation}
Let 
\[
\pi_\lambda:\A\to B(\fH), \quad \lambda\in \Lambda
\]
be an Exel--Loring approximation for $\pi$. Now, let $\Omega$ be the directed set consisting of pairs $(n,\lambda)$ where $n\in \bN$ and $\lambda\in \Lambda$. For $\omega=(n,\lambda)\in \Omega$, define a representation $\sigma_\omega:\A\to B(\fH_\sigma)$ as
\[
\sigma_\omega(a)=U_n \pi_{\lambda}(a) U_n^*, \quad a\in \A.
\]
It then follows easily from Equations \eqref{Eq:aue} and \eqref{Eq:aue*} that the set $\{\sigma_\omega:\omega\in \Omega\}$ contains an Exel--Loring approximation for $\sigma$. 
\end{proof}

Next, we show that restrictions to invariant subspaces preserve the existence of finite-dimensional approximations.

\begin{lemma}\label{L:approxinvsub}
 Let $\A$ be an operator algebra and let $\pi:\A\to B(\fH_\pi)$ be a representation. Let $\fH\subset \fH_\pi$ be a closed subspace. The following statements hold.
 \begin{enumerate}[{\rm (i)}]
 \item   Assume $\pi$ admits a finite-dimensional approximation and that $\fH$ is invariant for $\pi(\A)$. Then, the representation 
 \[
 a\mapsto \pi(a)|_\fH, \quad a\in \A
 \]
 admits a finite-dimensional approximation.

 \item   Assume $\pi$ admits an Exel--Loring approximation and that $\fH$ is reducing for $\pi(\A)$. Then, the representation 
 \[
 a\mapsto \pi(a)|_\fH, \quad a\in \A
 \]
admits an Exel--Loring approximation.
 \end{enumerate}

\end{lemma}
\begin{proof}
Clearly, we may assume without loss of generality that $\fH$ is infinite-dimensional. Throughout the proof, we let $\Lambda$ denote the directed set of triples $(\F,\fV,\epsilon)$ where $\epsilon>0$, $\F$ is a finite subset of the unit ball of $\A$, and $\fV$ is a finite subset  of $\fH$.

(i) By assumption, there is a net
\[
 \pi_\omega:\A\to B(\fH_\pi), \quad \omega\in \Omega
\]
of representations such that $\rC^*(\pi_\omega(\A))\fH_\pi$ is finite-dimensional and $(\pi_\omega(a))$ converges in the SOT to $\pi(a)$ for every $a\in \A$. Thus,
\[
 \langle \pi(a)\xi,\pi(b)\eta\rangle=\lim_{\omega}\langle \pi_\omega(a)\xi,\pi_\omega(b)\eta \rangle
 \]
 for every $a,b\in \A$ and $\xi,\eta\in \fH_\pi$.
 
 Next, fix an element $\lambda=(\F,\fV,\epsilon)\in \Lambda$.  
 Let $\eps'>0$ be the positive number obtained by applying \cite[Lemma 3.5.6]{dixmier1977} to the finite subset
\[
\{x,\pi(a)\xi:a\in \F,\xi\in \fV\}
\]
which lies in $\fH$, since $\fH$ is assumed to be invariant for $\pi(\A)$. 
Let $\omega_\lambda\in \Omega$ be chosen large enough such that
\[
| \langle \pi(a)\xi,\pi(b)\eta\rangle-\langle \pi_{\omega_\lambda}(a)\xi,\pi_{\omega_\lambda}(b)\eta \rangle |<\eps'
\]
for every $a,b\in \F, \xi,\eta\in \fV$.  Because $\fH$ is infinite-dimensional, there is an isometry $W_\lambda:\rC^*(\pi_{\omega_\lambda}(\A))\fH_\pi\to \fH$.  Using \cite[Lemma 3.5.6]{dixmier1977}, there is a unitary operator $U_\lambda\in B(\fH)$ such that
\[
 \|U_{\lambda} W_\lambda x-x\|<\epsilon
\]
and
\[
\| U_{\lambda} W_\lambda \pi_{\omega_\lambda}(a)x-\pi(a)x\|<\epsilon
\]
for every $a\in \F, x\in \fV$.  Put $V_\lambda=U_\lambda W_\lambda: \rC^*(\pi_{\omega_\lambda}(\A))\fH_\pi\to \fH$. For $a\in \F, x\in \fV$, we note that
\[
\| V_\lambda^* x-x\|\leq \|x-V_\lambda x\|<\eps
\]
whence
\begin{align*}
\| V_\lambda \pi_{\omega_\lambda}(a) V^*_\lambda x -\pi(a)x\|&=\|V_\lambda \pi_{\omega_\lambda}(a) V^*_\lambda x -V_\lambda \pi_{\omega_\lambda}(a)x\|+\| V_\lambda \pi_{\omega_\lambda}(a)x-\pi(a)x\|\\
&<\| \pi_{\omega_\lambda}(a) \|  \| V_\lambda^*x- x\|+\eps\\
&\leq (1+\|a\|)\epsilon\leq 2\eps.
\end{align*}
We thus conclude that the net $(V_\lambda \pi_{\omega_\lambda}(a) V^*_\lambda)_\lambda$ converges in the SOT to $\pi(a)|_{\fH}$ for every $a\in \A$. Moreover, for each $\lambda\in \Lambda$ we see that $\rC^*( V_\lambda \pi_{\omega_\lambda}(\A)V_\lambda^*)\fH$ is finite-dimensional since it is contained in $V_\lambda \rC^*(\pi_{\omega_\lambda}(\A))\fH_\pi$. Thus, $(V_\lambda \pi_{\omega_\lambda}(a) V^*_\lambda)_\lambda$ is the desired finite-dimensional approximation.

(ii) The argument is similar to the one above, with some minor modifications. By assumption, there is a net
\[
 \pi_\omega:\A\to B(\fH_\pi), \quad \omega\in \Omega
\]
of representations such that $\rC^*(\pi_\omega(\A))\fH_\pi$ is finite-dimensional. Moreover, for every $a\in \A$, in the SOT we have that $(\pi_\omega(a))$ converges $\pi(a)$ while $(\pi_\omega(a))^*$ converges $\pi(a)^*$. Thus,
\[
 \langle \pi(a)\xi,\pi(b)\eta\rangle=\lim_{\omega}\langle \pi_\omega(a)\xi,\pi_\omega(b)\eta \rangle
 \]
 \[
 \langle \pi(a)\xi,\pi(b)^*\eta\rangle=\lim_{\omega}\langle \pi_\omega(a)\xi,\pi_\omega(b)^*\eta \rangle
 \]
 \[
 \langle \pi(a)^*\xi,\pi(b)^*\eta\rangle=\lim_{\omega}\langle \pi_\omega(a)^*\xi,\pi_\omega(b)^*\eta \rangle
 \]
 for every $a,b\in \A$ and $\xi,\eta\in \fH_\pi$.

 Next, fix an element $\lambda=(\F,\fV,\epsilon)\in \Lambda$.  Let $\eps'>0$ be the positive number obtained by applying \cite[Lemma 3.5.6]{dixmier1977} to the finite subset
\[
\{x,\pi(a)\xi,\pi(a)^*\xi:a\in \F,\xi\in \fV\}
\]
which lies in $\fH$, since $\fH$ is assumed to be reducing for $\pi(\A)$. 
Let $\omega_\lambda\in \Omega$ be chosen large enough such that
\[
| \langle \pi(a)\xi,\pi(b)\eta\rangle-\langle \pi_{\omega_\lambda}(a)\xi,\pi_{\omega_\lambda}(b)\eta \rangle |<\eps',
\]
\[
| \langle \pi(a)\xi,\pi(b)^*\eta\rangle-\langle \pi_{\omega_\lambda}(a)\xi,\pi_{\omega_\lambda}(b)^*\eta \rangle |<\eps',
\]
and
\[
| \langle \pi(a)^*\xi,\pi(b)^*\eta\rangle-\langle \pi_{\omega_\lambda}(a)^*\xi,\pi_{\omega_\lambda}(b)^*\eta \rangle |<\eps',
\]
for every $a,b\in \F, \xi,\eta\in \fV$.  Arguing as in (i), we find an isometry 
\[
V_\lambda: \rC^*(\pi_{\omega_\lambda}(\A))\fH_\pi\to \fH
\]
with the property that
\[
 \|V_{\lambda} x-x\|<\epsilon
\]
and
\[
\| V_{\lambda}\pi_{\omega_\lambda}(a)x-\pi(a)x\|<\epsilon, \quad \| V_{\lambda} \pi_{\omega_\lambda}(a)^*x-\pi(a)^*x\|<\eps
\]
for every $a\in \F, x\in \fV$. As before, the net $(V_\lambda \pi_{\omega_\lambda}(a) V^*_\lambda)_\lambda$ is verified to be the desired Exel--Loring approximation.
\end{proof}

An important application of the preceding sequence of lemmas is the next tool.

\begin{theorem}\label{T:Voiculescu}
 Let $\A$ be a separable operator algebra. Let $\pi:\rC^*(\A)\to B(\fH_\pi)$ and $\sigma:\rC^*(\A)\to B(\fH_\sigma)$ be two $*$-representations. Assume that $\pi$ is injective and that $\pi|_\A$ admits an Exel--Loring approximation. Then, $\sigma|_\A$ also admits an Exel--Loring approximation. 
\end{theorem}
\begin{proof}
Because $\pi$ is injective and $\A$ is separable, we may find a separable closed subspace $\fH_\pi'\subset \fH_\pi$ which is reducing for $\pi(\rC^*(\A))$ and such that the restriction $\pi':\rC^*(\A)\to B(\fH_\pi')$ defined as
\[
\pi'(t)= \pi(t)|_{\fH'_\pi}, \quad t\in \rC^*(\A)
\]
is injective and non-degenerate.  By Lemma \ref{L:approxinvsub}, we see that $\pi'$ also admits an Exel--Loring approximation. Thus, we may assume that $\pi$ itself is injective, non-degenerate, and that $\fH_\pi$ is separable. 

Next, we may write $\sigma=\sigma'\oplus 0$, where $\sigma'$ is a non-degenerate $*$-representation which can, in turn, be decomposed as the direct sum of cyclic $*$-representations. Lemma \ref{L:directsum} implies that it suffices to establish the desired result for all these cyclic $*$-representations. Thus, we may simply assume that $\sigma$ is non-degenerate and that $\fH_\sigma$ is separable.

Consider then the non-degenerate $*$-representations $\Pi:\rC^*(\A)\to B(\H_\pi^{(\infty)})$ and $\Sigma: \rC^*(\A)\to B(\H_\sigma^{(\infty)})$ defined as
\[
\Pi(t)=\pi(t)^{(\infty)} \qand \Sigma(t)=\sigma(t)^{(\infty)} 
\]
for every $t\in \rC^*(\A)$. In particular, we see that $\Pi$ is injective, and that $\Pi(\rC^*(\A))$ and $\Sigma(\rC^*(\A))$ contain no compact operators. The non-degenerate $*$-representations $\Pi$ and $\Sigma\oplus \Pi$ are both injective, so that by \cite[Corollary II.5.6]{Davbook}  they are approximately unitarily equivalent.  By assumption, $\pi|_\A$ admits an Exel--Loring approximation, and thus so does $\Pi|_\A$ by Lemma \ref{L:directsum}. In turn, Lemma \ref{L:ELaue} implies that $(\Sigma\oplus \Pi)|_{\A}$ admits an Exel--Loring approximation. Finally, Lemma \ref{L:approxinvsub} implies that $\Sigma|_\A$ and $\sigma|_\A$ admit an Exel--Loring approximations as well.
\end{proof}

The final preliminary tool we require is the following, which provides Exel--Loring approximations for representations that extend to $*$-representations of some RFD $\rC^*$-cover.

\begin{lemma}\label{L:extEL}
Let $(\fB, \iota)$ be a $\rC^*$-cover of some operator algebra $\A$ such that $\fB$ is RFD.  Let $\pi$ be a representation of $\A$ such that $\pi\circ \iota^{-1}$ extends to a $*$-representation of $\fB$. Then, $\pi$ admits an Exel--Loring approximation.
\end{lemma}
\begin{proof}
Write $\pi:\A\to B(\fH_\pi)$. By assumption, there is a $*$-representation $\sigma:\fB\to B(\fH_\pi)$ such that $\sigma\circ \iota=\pi.$ Invoking \cite[Theorem 2.4]{EL1992}, we know that $\sigma$ admits an Exel--Loring approximation
\[
 \sigma_\lambda:\fB\to  B(\fH_\pi),\quad \lambda\in \Lambda.
\]
This means that $\sigma_\lambda(\fB)\fH_\pi$ is finite-dimensional for every $\lambda\in \Lambda$, and $(\sigma_\lambda(b))$ converges to $\sigma(b)$ in SOT for every $b\in\fB$.
For each $\lambda\in\Lambda$, let $\pi_\lambda=\sigma_\lambda\circ \iota$. Note first that
\[
\rC^*(\pi_\lambda(\A))\fH_\pi=\sigma_\lambda(\fB)\fH_\pi
\]
is finite-dimensional for every $\lambda\in \Lambda$. Next, fix $a\in \A$. Then, we find
\[
\pi_\lambda(a)=\sigma_\lambda(\iota(a)) \qand \pi_\lambda(a)^*=\sigma_\lambda(\iota(a)^*)
\]
for every $\lambda\in \Lambda$. We infer that $(\pi_\lambda(a))$ converges to $\pi(a)$ and $(\pi_\lambda(a)^*)$ converges to $\pi(a)^*$ in SOT. Consequently, $(\pi_\lambda)$ is an Exel--Loring approximation of $\pi$.
\end{proof}

\section{Characterizations of residual finite-dimensionality}\label{S:charRFD}
In this section, we use finite-dimensional approximations and the related tools developed in Section \ref{S:fdimapprox} to analyze residual finite-dimensionality of general (that is, possibly non-selfadjoint) operator algebras. We start with a characterization of residual finite-dimensionality. 

\begin{theorem}\label{T:ELOA}
 Let $\A$ be an operator algebra. The following are equivalent.
 \begin{enumerate}[{\rm (i)}]
  \item The algebra $\A$ is RFD.
  \item Every extremal representation of $\A$ admits a finite-dimensional approximation. 
  \item Every extremal representation of $\A$ admits an Exel--Loring approximation. 
  \item Every coextension-extremal representation of $\A$ admits a finite-dimensional approximation.
 \end{enumerate}
\end{theorem}
\begin{proof}
(iii) $\Rightarrow$ (ii): This is trivial.

(ii) $\Rightarrow$ (i): By virtue of \cite[Theorems 1.1 and 1.2]{dritschel2005} combined with \cite[Proposition 2.5]{DS2018}, we may find a set $\S$ of extremal representations of $\A$ with the property that the map $\bigoplus_{\pi\in \S}\pi$ is completely isometric. Let $d\in \bN$ and $A\in \bM_d(\A)$ with $\|A\|=1$. Let $\eps>0$. There is a representation $\pi:\A\to B(\fH_\pi)$ in $\S$ such that 
\[
\|\pi^{(d)}(A)\|\geq (1-\eps).
\]
On the other hand, by assumption we see that $\pi$ admits a finite-dimensional approximation $(\pi_\lambda)_{\lambda\in \Lambda}$. Thus, there is $\lambda\in \Lambda$ with the property that
\[
\|\pi_\lambda^{(d)}(A)\|\geq (1-2\eps).
\]
Let $\fH_\lambda=\rC^*(\pi_\lambda(\A))\fH_\pi$. This is a finite-dimensional reducing subspace for $\pi_\lambda$. Define $\rho_\lambda:\A\to B(\fH_\lambda)$ to  be the representation
\[
\rho_\lambda(a)=\pi_\lambda(a)|_{\fH_\lambda}, \quad a\in \A.
\]
We find
\[
\pi_\lambda(a)=\rho_\lambda(a)\oplus 0, \quad a\in\A
\]
whence $\pi_\lambda^{(d)}(A)$ is unitarily equivalent to $\rho_\lambda^{(d)}(A)\oplus 0$, so that
\[
\|\rho_\lambda^{(d)}(A)\|\geq (1-2\eps).
\]
We conclude that $\A$ is RFD.

(i) $\Rightarrow$ (iii): It follows readily from the definition that because $\A$ is RFD, there is a $\rC^*$-cover $(\fB,\iota)$ of $\A$ such that $\fB$ is an RFD $\rC^*$-algebra. Let $\pi$ be an extremal representation of $\A$.  By virtue of Theorem \ref{T:extreps}, we may find a $*$-representation $\sigma$ of $\fB$ such that $\sigma\circ \iota=\pi$. It then follows from Lemma \ref{L:extEL} that $\pi$ admits an Exel--Loring approximation.

(iv) $\Rightarrow$ (ii): This is trivial.

(ii) $\Rightarrow$ (iv): Let $\pi$ be a coextension-extremal representation of $\A$. By Theorem \ref{T:extreps}, we see that there is an extremal representation $\sigma$ of $\A$ which is an extension of $\pi$. Since $\sigma$ is assumed to admit a finite-dimensional approximation, Lemma \ref{L:approxinvsub} implies that $\pi$ admits a finite-dimensional approximation as well.
\end{proof}

Let $\A$ be an operator algebra and let $\pi$ be an extremal representation of $\A$. Let $(\rC^*_e(\A),\eps)$ denote the $\rC^*$-envelope of $\A$. By Theorem \ref{T:extreps}, we see that $\pi\circ \eps^{-1}$  extends to a $*$-representation of $\rC^*_e(\A)$. Unfortunately, the converse statement fails in general, in the sense that there are $*$-representations of $\rC^*_e(\A)$ whose restriction to $\eps(\A)$ is not extremal; in other words, there are operator algebras which are not \emph{hyperrigid}  \cite{arveson2011}. We mention that this kind of pathology occurs even in the classical setting of uniform algebras on compact metric spaces \cite[page 42]{phelps2001}. 

In light of this discussion, for non-hyperrigid RFD operator algebras, Theorem  \ref{T:ELOA} is inadequate to deal with general $*$-representations of the $\rC^*$-envelope. We remedy this shortcoming in the next result, at least in the separable setting.

\begin{theorem}\label{T:RFDcharC*env}
Let $\A$ be a separable operator algebra and let $(\rC^*_e(\A), \eps)$ denote its $\rC^*$-envelope. Then $\A$ is RFD if and only if $\sigma\circ \eps$ admits an Exel--Loring approximation for every $*$-representation $\sigma$ of $\rC^*_e(\A)$.
\end{theorem}
\begin{proof}
Assume first that $\sigma\circ \eps$ admits an Exel-Loring approximation for every $*$-representation $\sigma$ of $\rC^*_e(\A)$. 
Let $\pi$ be an extremal representation of $\A$. By Theorem \ref{T:extreps}, there is a $*$-representation $\theta$ of $\rC^*_e(\A)$ extending $\pi\circ \eps^{-1}$. By assumption, $\pi=\theta\circ \eps$ admits an Exel--Loring approximation, so we conclude that $\A$ is RFD by virtue of Theorem \ref{T:ELOA}.

Conversely, assume that $\A$ is RFD. 
Let $\A^1$ denote the unitization of $\A$ (see \cite[Paragraph 2.1.11]{BLM2004}), and let $(\rC^*_e(\A^1), \iota)$ denote the $\rC^*$-envelope of $\A^1$. Combining \cite[Theorem 7.1]{arveson2008} and \cite[Proposition 4.4]{arveson2011}, we may find a separable Hilbert space $\fH_\pi$ along with a unital $*$-representation $\pi:\rC^*_e(\A^1)\to B(\fH_\pi)$ with the property that $\pi|_{\iota(\A^1)}$ is completely isometric and extremal. In turn, a well-known property of the $\rC^*$-envelope \cite[Proposition 2.2.4]{arveson1969} implies that $\pi$ is injective on $\rC^*_e(\A^1)$. 
Note also that $\pi\circ \iota$ is an extremal representation of $\A^1$ by \cite[Theorem 2.1.2]{arveson1969}. 

Observe next that
\[
\rC^*_e(\A^1)=\rC^*(\iota(\A^1))=\rC^*(\iota(\A))+\bC I.
\] 
Thus, as explained in \cite[page 15]{arveson1976inv}, we may find two $*$-representations $\pi_0,\pi_1$ of $\rC^*_e(\A^1)$ such that $\pi_0(\rC^*(\iota(\A)))$ is non-degenerate, $\rC^*(\iota(\A))\subset \ker \pi_1$ and $\pi=\pi_0\oplus \pi_1$.
Because $\pi$ is injective, we see that $\pi_0$ must be injective on $\rC^*(\iota(\A))$.
By \cite[Paragraph 4.3.4]{BLM2004}, we know that we may choose $\rC^*_e(\A)= \rC^*(\iota(\A))$ and $\eps=\iota|_{\A}$. Furthermore,  \cite[Proposition 2.5]{DS2018} implies that $\pi_0|_{\iota(\A)}=\pi|_{\iota(\A)}$ is extremal since $\pi|_{\iota(\A^1)}$ is.  It follows from Theorem \ref{T:ELOA} that $\pi_0|_{\iota(\A)}$ admits an Exel--Loring approximation, and thus so does $\pi_0\circ \iota|_\A=\pi_0\circ \eps$. An application of Theorem \ref{T:Voiculescu} now yields the desired statement.
\end{proof}

As mentioned above, in the case of RFD $\rC^*$-algebras, it is known that every representation admits an Exel--Loring approximation \cite{EL1992}. Theorems \ref{T:ELOA} and \ref{T:RFDcharC*env} make similar claims in the non-selfadjoint context, but only for certain special representations. It is then natural to wonder when can the existence of Exel--Loring approximations be shown for general representations of RFD operator algebras. As we show next, showing that all representations admit an Exel--Loring approximation is in fact equivalent to a condition which is a priori stronger, namely the residual finite-dimensionality of the maximal $\rC^*$-cover.

\begin{theorem}\label{T:ELC*max}
 Let $\A$ be an operator algebra and let $(\rC^*_{\max}(\A),\upsilon)$ denote its maximal $\rC^*$-cover. Then, the following statements are equivalent.
 \begin{enumerate}[{\rm (i)}]
 
 \item The $\rC^*$-algebra  $\rC^*_{\max}(\A)$ is RFD.
 
 \item Every representation of $\A$ admits an Exel--Loring approximation.
 
 \item Let $\fH$ be a Hilbert space and assume that $\rC^*_{\max}(\A) \subset B(\fH)$. Then, the representation $\upsilon:\A\to B(\fH)$ admits an Exel--Loring approximation.

\end{enumerate}
 \end{theorem}
\begin{proof}
(i) $\Rightarrow$ (ii): Assume that $\rC^*_{\max}(\A)$ is RFD. Let $\pi:\A\to B(\fH_\pi)$ be a representation. Then, there is a $*$-representation $\widehat\pi:\rC^*_{\max}(\A)\to B(\fH_\pi)$ such that $\widehat\pi\circ \upsilon=\pi$. It thus follows from Lemma \ref{L:extEL} that $\pi$ admits an Exel--Loring approximation.

(ii) $\Rightarrow$ (iii): This is trivial.

(iii) $\Rightarrow$ (i): Assume that $\upsilon$ has an Exel--Loring approximation 
\[
\pi_\lambda:\A\to B(\fH), \quad \lambda\in \Lambda.
\]
Note that linear combinations of words with letters in $\upsilon(\A)\cup \upsilon(\A)^*$ are clearly dense in $\rC^*_{\max}(\A)$, since $\rC^*_{\max}(\A)=\rC^*(\upsilon(\A))$. Therefore, to show that the latter $\rC^*$-algebra is RFD it suffices to fix such an element $s$ with $\|s\|=1$ along with $\eps>0$ and to construct a finite-dimensional $*$-representation $\rho$ of $\rC^*_{\max}(\A)$ such that $\|\rho(s)\|\geq 1-\eps$.

For each $\lambda\in \Lambda$, we obtain a $*$-representation $\widehat{\pi}_\lambda:\rC^*_{\max}(\A)\to B(\fH)$ with the property that $\widehat{\pi}_\lambda\circ \upsilon=\pi_\lambda$. 
Because multiplication is jointly continuous in the SOT on bounded sets, it is readily verified that $(\widehat{\pi}_\lambda(s))$ converges to $s$ in the SOT. In particular, there exists $\lambda\in \Lambda$ such that $\|\widehat{\pi}_\lambda(s)\|\geq 1-\eps$. Note now that the space $\fH_\lambda=\widehat{\pi}_\lambda(\rC^*_{\max}(\A))\fH$ is finite-dimensional since
\[
\widehat{\pi}_\lambda(\rC^*_{\max}(\A))\fH=\rC^*(\widehat{\pi}_\lambda(\upsilon(\A))\fH=\rC^*(\pi_\lambda(\A))\fH
\]
and the latter is finite-dimensional by construction.
Thus, we may define a finite-dimensional $*$-representation $\rho:\rC^*_{\max}(\A)\to B(\fH_\lambda)$ as
\[
\rho(t)=\widehat{\pi}_\lambda(t)|_{\fH_\lambda}, \quad t\in \rC^*_{\max}(\A).
\]
It follows that
\[
\widehat{\pi}_\lambda(t)=\rho(t)\oplus 0, \quad t\in \rC^*_{\max}(\A).
\]
so in particular $\rho$ has the required property.
\end{proof}

In light of Theorems \ref{T:ELOA} and \ref{T:ELC*max}, we see that the existence of finite-dimensional and Exel--Loring approximations is at the heart of Question \ref{Q:rfdC*max}.

To illustrate some of the difficulties associated with this question, we remark that the approximations from Theorem \ref{T:ELC*max} can be rather complicated, even for very transparent choices of representations. For instance, let $\A\subset B(\fH)$ be an operator algebra such that $\rC^*_{\max}(\A)$ is RFD. Let $\fK\subset \fH$ be an invariant subspace for $\A$. Consider the representation $\pi:\A\to B(\fK)$ defined as $\pi(a)=a|_{\fK}$ for every $a\in \A$. Then, by Theorem \ref{T:ELC*max}, we know that $\pi$ admits an Exel--Loring approximation. But this approximation does not necessarily arise as compressions of $\A$ to finite-dimensional subspaces of $\fK$, as the following concrete example shows. Recall that a closed subspace $\fS\subset \fH$ is said to be \emph{semi-invariant} for $\A$ if the map
\[
a\mapsto P_{\fS} a|_{\fS}, \quad a\in \A
\]
is multiplicative (see \cite{sarason1965}). A standard verification reveals that $\fS$ is semi-invariant if and only if $\fS=\fM\ominus \fN$ for some closed subspaces $\fM,\fN\subset \fH$ that are invariant for $\A$. 

\begin{example}\label{E:seminv}
Let $\bD\subset \bC$ denote the open unit disc. Let $\mathrm{A}(\bD)$ be the classical disc algebra acting on the Hardy space $H^2$. Let $H^\infty$ denote the Banach algebra of bounded holomorphic functions on $\bD$. In this example, we will freely use well-known facts about the structure of these objects; the reader may consult \cite{hoffman1988} for details. Let $\theta\in H^\infty$ be the singular inner function corresponding to the point mass at $1$, i.e.
 \[
  \theta(z)=\exp\left(\frac{z+1}{z-1} \right), \quad z\in \bD.
 \]
Let $\fK_\theta=H^2\ominus \theta H^2$ and let $\pi:\mathrm{A}(\bD) \to B(\fK_\theta)$  be the unital representation defined as
\[
\pi(a)=P_{\fK_\theta}a|_{\fK_\theta}, \quad a\in \rA(\bD).
\] 
Note that $\rC^*_{\max}(\mathrm{A}(\bD))$ is RFD by \cite[Example 3]{CR2018rfd}. Therefore, $\pi$ admits an Exel--Loring approximation by Theorem \ref{T:ELC*max}. Nevertheless, as we show next, $\fK_\theta$ contains no finite-dimensional subspace which is semi-invariant for $\pi(\rA(\bD))$.

To see this,  assume that $\fS\subset \fK_\theta$  is such a subspace. We may find closed subspaces $\fM,\fN\subset \fK_\theta$ which are invariant for $\pi(\mathrm{A}(\bD))$ such that $\fS=\fM\ominus \fN$. Then, $\fM\oplus \theta H^2$ and $\fN\oplus \theta H^2$ are also closed and invariant for $\rA(\bD)$.  By Beurling's theorem, there are distinct inner functions  $\phi$ and $\psi$ in $H^\infty$ such that
\[
\phi H^2=\fM\oplus \theta H^2, \quad \psi H^2=\fN\oplus \theta H^2.
\]
Because $\theta \in \phi H^2\cap \psi H^2$, we see that $\phi$ and $\psi$ must be singular inner functions with $\phi H^2\subset \psi H^2$ or $\psi H^2\subset \phi H^2$. The first possibility would imply $\fM \subset \fN$ and hence $\fS=\{0\}$ contrary to our assumption. Thus, we must have then that $\psi H^2\subset \phi H^2$ and
\[
 \fS=\fM\ominus \fN=\phi H^2\ominus \psi H^2.
\]
We infer that $\fS$ is infinite-dimensional (this follows for instance  from \cite[Exercise III.1.11]{bercovici1988}). 
 \qed
\end{example}

\section{Constructing Exel--Loring approximations via semigroup coactions}\label{S:coactions}

Motivated by Question \ref{Q:rfdC*max}, in this section we aim to identify concrete instances of RFD operator algebras $\A$ with the property that $\rC^*_{\max}(\A)$ is RFD. Our approach is predicated on Theorem \ref{T:ELC*max}, in that we aim to produce Exel--Loring approximations for arbitrary representations. The unifying theme throughout this section will be that of semigroups and their coactions. For more on semigroups and their C*-algebras, we recommend \cite[Chapter 5]{CELY17} and \cite{li2012, li2013, li2017}.

Let $P$ be a discrete semigroup. Assume that $P$ is \emph{cancellative}, so that two elements $p,q\in P$ satisfy $p=q$ whenever there is $r\in P$ with $rp=rq$ or $pr=qr$. For $p\in P$, define
\[
\R_p=\{r\in P:p=qr \text{ for some } q\in P\}
\]
and
\[
\L_p=\{q\in P:p=qr \text{ for some } r\in P\}.
\]
Let $r\in \R_p$. By definition, we find $q_r\in \L_p$ such that $p=q_rr$. Since $P$ is cancellative, this element $q_r$ is uniquely determined. Likewise, given $q\in \L_p$ there is a unique $r_q\in \R_p$ such that $p=qr_q$. We conclude that there is a bijection between $\R_p$ and $\L_p$. 
We say that the cancellative semigroup $P$ has the \emph{finite divisor property} (or FDP) if for every $p\in P$ the sets $\R_p$ and $\L_p$ are finite. Clearly the finite divisor property is inherited by subsemigroups. Semigroups with FDP include important examples such as free monoids, left-angled Artin monoids, Baumslag--Solitar monoids, braid monoids, Thompson's monoid and more.

Next, let $\lambda:P\to B(\ell^2(P))$ denote the left regular representation. For each $p\in P$, we denote by $e_p\in \ell^2(P)$ the characteristic function of $\{p\}$. Thus, $\{e_p\}_{p\in P}$ is an orthonormal basis for $\ell^2(P)$. Because $P$ is left cancellative, it is easily seen that $\lambda_p$ is an isometry for every $p\in P$. Furthermore, for $p\in P$ we see that
\begin{equation}\label{Eq:reg}
\lambda_p^*(e_r) = \begin{cases}
e_q &\mbox{if } r = pq \\
0 & \mbox{if } r \notin pP.
\end{cases}
\end{equation}
Let $\A_r(P)\subset B(\ell^2(P))$ denote the norm closed operator algebra generated by the image of $\lambda$, and let $\rC^*_r(P)$ denote the $\rC^*$-algebra generated by the image of $\lambda$. The developments in this section hinge on the following fact.

\begin{proposition}\label{P:ArP}
Let $P$ be a cancellative semigroup with FDP. For every finite subset $F\subset P$, there is a representation $\pi_F:\A_r(P)\to B(\ell^2(P))$ such that
\[
P\setminus \left( \cup_{p\in F}\cup_{r\in \R_p}\L_r\right)=   \lambda^{-1}(\ker \pi_F).
\] 
Moreover, the net $(\pi_F)_F$ is an Exel--Loring approximation for the identity representation of $\A_r(P)$. In particular,  $\A_r(P)$ is RFD.
\end{proposition}
\begin{proof}
For each $p\in P$, the subset $\R_p$ is finite by the assumption on the semigroup. Correspondingly, let $\X_p\subset \ell^2(P)$ denote the finite-dimensional subspace spanned by 
$
\{e_r:r\in \R_p\}.
$ Note now that if $r\in \R_p$, then $\R_r\subset \R_p$. Invoking \eqref{Eq:reg}, we find for $p,s\in P$ and $r\in \R_p$ that
\begin{equation}\label{Eq:coinv}
\lambda_s^* e_r \in \{0\}\cup \{e_t: t \in \R_r\}\subset \X_p
\end{equation}
and that
\begin{equation}\label{Eq:kernel}
\lambda_s^*e_r\neq 0 \quad \text{ if and only if } \quad s\in \L_r.
\end{equation}
We conclude from \eqref{Eq:coinv} that $\X_p$ is coinvariant for $\A_r(P)$ for every $p\in P$. 

Next, for every finite subset $F\subset P$, we let $\Y_F=\sum_{p\in F}\X_p$. Clearly, $\Y_F$ is a finite-dimensional subspace which is coinvariant for $\A_r(P)$. Let $Q_F\in B(\ell^2(P))$  denote the finite-rank orthogonal projection onto $\Y_F$. Thus, the map $\pi_F:\A_r(P)\to B(\ell^2(P))$ defined as
\[
\pi_F(a)=Q_FaQ_F, \quad a\in \A_r(P)
\]
is a representation. It readily follows from the fact that
\[
\cup_{p\in P}\R_p=P
\]
that the net $(\pi_F)$ is an Exel--Loring approximation for the identity representation of $\A_r(P)$, and that $\A_r(P)$ is RFD. Finally, \eqref{Eq:kernel} implies that $\pi_F(\lambda_s)\neq 0$ if and only if $s\in \cup_{p\in F}\cup_{r\in \R_p}\L_r$, which is equivalent to 
\[
P\setminus \left( \cup_{p\in F}\cup_{r\in \R_p}\L_r\right)=   \lambda^{-1}(\ker \pi_F).
\] 
\end{proof}

Coactions of discrete groups on general operator algebras were introduced in \cite{DKKLL+} as a generalization of coactions of discrete groups on C*-algebras \cite{Quigg1996}. The following is a semigroup variant of \cite[Definition 3.1]{DKKLL+}, adapted to our context. 

\begin{definition} \label{d:coaction}
Let $\A$ be an operator algebra, $P$ be a cancellative discrete semigroup, and $\delta : \A \rightarrow \A \otimes_{\min} \A_r(P)$ a completely isometric homomorphism. For each $p\in P$ denote by $\A^{\delta}_p$ the $p$-th \emph{spectral subspace} according to $\delta$, given by
\[
\A_p^{\delta} = \{ a \in \A \ | \ \delta(a) = a \otimes \lambda_p \}. 
\]
We say that $\delta$ is a \emph{coaction of $P$ on $\A$} if $\sum_{p\in P}\A^{\delta}_p$ is norm dense in $\A$.

\end{definition}
We remark here that in principle, the role of $\A_r(P)$ in the definition of a semigroup coaction can be fulfilled by a norm closed algebra generated by any contractive representation of $P$. However, we will not need this added generality for our purposes. A useful property of semigroup coactions that we require is the following.

\begin{lemma}\label{L:coactionchar}
Let $P$ be a cancellative discrete semigroup with the property that there is a character $\chi: \A_r(P)\to \bC$ such that $\chi(\lambda_p)=1$ for every $p\in P$. Let $\A$ be an operator algebra and let $\pi:\A\to B(\fH_\pi)$ be a representation. If $\delta$ is a coaction of $P$ on $\A$, then $(\pi\otimes \chi)\circ \delta $ is unitarily equivalent to $\pi$.
\end{lemma}
\begin{proof}
For $p\in P$ and $a\in \A^{\delta}_p$, we find
\[
(\pi\otimes \chi)\circ \delta(a)=(\pi\otimes \chi)(a\otimes \lambda_p)=\pi(a)\otimes 1.
\]
The density of the sum of the spectral subspaces thus implies that $(\pi\otimes \chi)\circ \delta $ is unitarily equivalent to $\pi$.
\end{proof}

Next, let $\F$ denote the directed set of finite subsets of $P$. Proposition \ref{P:ArP} yields the existence of a net $(\pi_F)_{F\in \F}$ which is an Exel--Loring approximation for the identity representation of $\A_r(P)$. For each $F\in \F$, we let  $\Q_F$ be the quotient of $\A$ by the norm closure of the ideal generated by $\sum_{p \in \lambda^{-1}(\ker \pi_F)} \A_p^{\delta}$. 

\begin{definition}
Let $\A$ be an operator algebra, $P$ be a cancellative discrete semigroup with FDP, and $\delta$ a coaction of $P$ on $\A$. We say that $\delta$ is \emph{residually finite--dimensional} (RFD) if the $\rC^*$-algebra $\rC^*_{\max}(\Q_F)$ is RFD for each $F\in \F$.
\end{definition}

We now take our first crucial step towards answering Question \ref{Q:rfdC*max} for a class of operator algebras with semigroup coactions.
 
\begin{proposition}\label{P:ELapproxcoactions}
 Let $\A$ be an operator algebra and let $P$ be a cancellative semigroup with FDP.  Assume that there exists an RFD coaction $\delta$ of $P$ on $\A$. Let $\pi:\A\to B(\fH_\pi)$ be a representation. Then, the representation $(\pi\otimes \id)\circ \delta$ admits an Exel--Loring approximation.
 \end{proposition}

 \begin{proof}
 Let $F\subset P$ be a finite subset, and let $q_F :\A\to \Q_F$ denote the natural quotient map. If $p\in \lambda^{-1}(\ker \pi_F)$ and $a\in \A^{\delta}_p$, then
\[
(\pi\otimes\pi_F)\circ \delta(a)=\pi(a)\otimes \pi_F(\lambda_p)=0.
\]
We conclude that
\[
\sum_{p\in \lambda^{-1}(\ker \pi_F)}\A^{\delta}_p \subset \ker (\pi\otimes \pi_F)\circ \delta. 
\]
Thus, there is a representation $\rho_F :\Q_F \to B(\fH_\pi\otimes \ell^2(P)))$ with the property that 
\[
(\pi\otimes \pi_F)\circ \delta=\rho_F \circ q_F.
\]
Since the coaction $\delta$ is assumed to be RFD, we see that $\rC^*_{\max}(\Q_F)$ is RFD, so by Theorem \ref{T:ELC*max} the representation $\rho_F$ admits an Exel--Loring approximation 
\[
\rho_{F,\mu}:\Q_F \to B(\fH_\pi\otimes \ell^2(P)), \quad  \mu\in \Omega_F. 
\]
In particular,
\[
\rC^*((\rho_{F,\mu}\circ q_F)(\A))(\fH_\pi \otimes\ell^2(P))=\rC^*(\rho_{F,\mu}(\Q_F))(\fH_\pi\otimes\ell^2(P))
\]
is finite-dimensional for every $\mu\in \Omega_F$. 

For every $p\in P$ and every $a\in \A^{\delta}_p$, given $\xi\in \fH_\pi$ and $\eta\in \ell^2(P)$ we obtain
\begin{align*}
\lim_F \lim_\mu (\rho_{F,\mu}\circ q_F)(a)(\xi\otimes \eta)&=\lim_F(\rho_F\circ q_F) (a)(\xi\otimes \eta)\\
&=\lim_F((\pi\otimes \pi_F)\circ \delta) (a)(\xi\otimes \eta)\\
&=\lim_F \pi(a)\xi\otimes \pi_F (\lambda_p)\eta\\
&=\pi(a)\xi\otimes \lambda_p \eta\\
&=(\pi(a)\otimes \lambda_p)(\xi\otimes \eta)\\
&=((\pi\otimes \id)\circ \delta)(a)(\xi\otimes\eta)
\end{align*}
and likewise
\[
\lim_F \lim_\mu (\rho_{F,\mu}\circ q_F)(a)^*(\xi\otimes \eta)=((\pi\otimes \id)\circ \delta)(a)^*(\xi\otimes\eta).
\]
A standard density argument then reveals that there is a net in 
\[
\{\rho_{F,\mu}\circ q_F:F\subset P \text{ finite}, \mu\in \Omega_F\}
\]
which forms an Exel--Loring approximation for $(\pi\otimes \id)\circ \delta$.
 \end{proof}

The following is the connection between coactions of semigroups on RFD operator algebras, and RFD maximal C*-covers.

\begin{theorem} \label{T:coactionC*maxRFD}
Let $P$ a countable cancellative semigroup with FDP such that there is a character $\chi:\A_r(P)\to \bC$ with $\chi(\lambda_p)=1$ for every $p\in P$. Let $\A$ be a separable operator algebra that admits an RFD coaction of $P$. Then, $\rC^*_{\max}(\A)$ is RFD.
\end{theorem}

\begin{proof}
Recall that $\upsilon:\A\to \rC^*_{\max}(\A)$ is the completely isometric homomorphism such that $(\rC^*_{\max}(\A),\upsilon)$ is the maximal $\rC^*$-cover of $\A$.  By Theorem \ref{T:ELC*max}, it suffices to fix an injective $*$-representation $\rho:\rC^*_{\max}(\A)\to B(\fH)$ and to show that the representation $\rho\circ\upsilon:\A\to B(\fH)$ admits an Exel--Loring approximation.  

To see this, let $\delta$ be an RFD coaction of $P$ on $\A$. By injectivity of the minimal tensor product, the map
\[
\gamma=\upsilon\otimes \id:\A\otimes_{\min}\A_r(P)\to \rC^*_{\max}(\A)\otimes_{\min} \rC^*_r(P)
\]
is a completely isometric homomorphism. Let $\widehat\delta:\rC^*_{\max}(\A)\to \rC^*_{\max}(\A)\otimes_{\min} \rC^*_r(P)$ be the $*$-representation such that $\widehat\delta\circ \upsilon=\gamma \circ \delta$. 

Now, it follows from  Lemma \ref{L:coactionchar} that $(\upsilon\otimes \chi)\circ\delta$ is unitarily equivalent to $\upsilon$. Let $\widehat\chi: \rC^*_r(P)\to \bC$ be a state extending $\chi$. Since $\lambda_p$ is an isometry  and $\chi(\lambda_p)=1$, it follows that $\lambda_p$ lies in the multiplicative domain of $\chi$ for every $p\in P$. Thus, $\widehat\chi$ is a character,  which in turn implies that $(\rho\otimes \widehat\chi)\circ \widehat\delta$ is unitarily equivalent to $\rho$. In particular, because $\rho$ is injective this forces $\widehat\delta$ to be an injective $*$-representation and thus so is $(\rho\otimes \id)\circ \widehat\delta$.
Proposition \ref{P:ELapproxcoactions} now yields an Exel--Loring approximation for $((\rho\circ \upsilon)\otimes \id)\circ \delta$. But we find
\begin{align*}
((\rho\circ \upsilon)\otimes \id)\circ \delta&=(\rho\otimes \id)\circ \widehat\delta\circ \upsilon
\end{align*}
so that $(\rho\otimes \id)\circ \widehat\delta|_{\upsilon(\A)}$ admits an Exel--Loring approximation. Since $P$ is countable, the algebra $\A_r(P)$ is separable and we may invoke Theorem \ref{T:Voiculescu} to conclude that $\rho|_{\upsilon(\A)}$ admits an Exel--Loring approximation, and hence the same is true for $\rho\circ \upsilon$.
 \end{proof}

The astute reader may notice that the condition that $\A_r(P)$ admit an appropriate character in the previous theorem is stronger than what is actually needed in the proof. Indeed, therein we only need to know that the coaction $\delta$ admits an extension to an injective $*$-representation $\widehat\delta: \rC^*_{\max}(\A)\to \rC^*_{\max}(\A)\otimes_{\min} \rC^*_r(P)$. Unfortunately, at present we do not know how to verify this condition without assuming the existence of a character, or when the coaction is trivial.

\section{Algebras of semigroups and algebras of functions}\label{SS:semigroup}

In this section, we exhibit several classes examples of operator algebras to which Theorem \ref{T:coactionC*maxRFD} can be applied. These include algebras arising from semigroups and algebras of functions. 

\subsection{Semigroup operator algebras}
We recall some background material on semigroups. Let $P$ be a cancellative semigroup. By a right ideal of $P$ we mean a subset $X \subset P$ which is closed under right multiplication. Given a subset $X\subset P$ and an element $p\in P$, we define
 \[
 pX = \{ px : x\in X\} \qand p^{-1}X = \{x : px \in X\}.
 \]
Further, we denote by $\J_P$ the smallest collection of right ideals that contains $\varnothing$ and $P$, and such that $X\in \J_P$ implies $pX,p^{-1}X\in \J_P$.

\begin{hypothesis} 
Henceforth, we will assume that a semigroup $P$ is \emph{unital} and embeds as a unital subsemigroup of a group $G$.
\end{hypothesis}

We emphasize here that for $p\in P$ and $X\subset P$, $p^{-1}X$ is always understood to be a subset of $P$. In other words, $p^{-1}X$ is the intersection of $P$ with the set
\[
\{p^{-1}x:x\in X\}\subset G.
\]
%

The \emph{full semigroup $\rC^*$-algebra}, denoted by $\rC_s^*(P)$,  is the universal $\rC^*$-algebra generated by isometries $\{v_p:p\in P\}$ and self-adjoint projections $\{e_X:X\in \J_P\}$ such that
\begin{enumerate}
\item $e_{\varnothing} = 0$;

\item $v_p v_q = v_{pq}$ for $p,q\in P$;

\item whenever $p_1,q_1,...,p_n, q_n \in P$ satisfy $p_1^{-1} q_1\ldots p_n^{-1}q_n = e$ in $G$, then
$$
v_{p_1}^*v_{q_1} \ldots v_{p_n}^*v_{q_n} = e_{q_n^{-1}p_n\ldots q_1^{-1}p_1 P}.
$$
\end{enumerate}
Furthermore, we let $\fD_s(P)\subset \rC^*_s(P)$ denote the $\rC^*$-subalgebra generated by $\{e_X:X\in \J_P\}$.

The defining relations above can be realized inside of $\rC^*_r(P)$, as we explain next. For added clarity, thoughout this section we denote the left regular representation of $P$ by $\lambda^P$. Given a subset $X\subset P$, we let $1^P_X\in B(\ell^2(P))$ denote the operator of multiplication by the indicator function of $X$. Clearly, this is a self-adjoint projection. We denote by $\fD_r(P)$ the C*-subalgebra of $\rC^*_r(P)$ generated
\[
\{1^P_X:X\in \J_P\}.
\]
Now, it follows from \cite[Lemma 3.1]{li2012} that the families $\{\lambda^P_p:p\in P\}$ and $\{1^P_X:X\in \J_P\}$ inside $B(\ell^2(P))$ satisfy the defining relations (i)--(iii) above. Hence, by universality we obtain a surjective $*$-representation $\lambda^P_* : \rC_s^*(P) \rightarrow \rC^*_r(P)$ such that 
\[
\lambda^P_*(p) = \lambda^P_p, \quad p\in P
\]
 and 
 \[
 \lambda^P_*(e_X) = 1_X, \quad X\in \J_P.
 \] 
By \cite[Lemma 3.11]{li2012} there is a faithful conditional expectation $E^P_r : \rC^*_r(P) \rightarrow \fD_r(P)$. 
 
A cancellative semigroup $P$ 
is said to be \emph{independent} if whenever $X\in \J_P$ can be written as
\[
X= \bigcup_{i=1}^n X_i
\]
for some $X_1,\ldots,X_n\in \J_P$, then necessarily we have $X=X_j$ for some $1\leq j\leq n$. This property will play a crucial role below. Fortunately, many semigroups from the literature automatically satisfy independence. For instance, all right LCM semigroups are independent by \cite[Lemmas 6.31 and 6.32]{li2017}. 

Assume now that $P$ embeds into a group and is independent. By \cite[Lemma 3.13]{li2012} there is a conditional expectation $E^P_s : \rC_s^*(P) \to \fD_s(P)$ so that $\lambda^P_* \circ E^P_s = E^P_r \circ \lambda^P_*$ and
\begin{equation}\label{Eq:CE-kern}
\ker(\lambda^P_*) = \{  a \in \rC^*_s(P) : E^P_s(a^*a) = 0 \}.
\end{equation}
This fact will be used in the proof of the next result, which shows, roughly speaking, that certain semigroup homomorphisms induce $*$-representations on the reduced $\rC^*$-algebras.

\begin{proposition} \label{p:injectivity}
Let $P$ and $Q$ be semigroups embedded into some groups $G$ and $H$ respectively. Suppose that $P$ is independent and that $\psi : G \rightarrow H$ is an injective homomorphism such that $\psi(P) \subset Q$ and 
\begin{equation}\label{Eq:missingpsi}
\psi(p^{-1} X) = \psi(p)^{-1} \psi(X) \quad \text{for every } p\in P, X\in \J_P.
\end{equation}
Then, there is a $*$-representation $\Psi_r : \rC^*_r(P) \rightarrow \rC^*_r(Q)$ such that
\[
\Psi_r(\lambda^P_p)=\lambda^Q_{\psi(p)}, \quad p\in P.
\]
\end{proposition}

\begin{proof}
For each $p\in P$, define the isometry $w_p=\lambda^Q_{\psi(p)}\in B(\ell^2(Q))$, and for each $X\in \J_P$, define the projection $f_X=1^Q_{\psi(X)Q}\in B(\ell^2(Q))$. 
It is readily seen that the collections $\{w_p:p\in P\}$ and $\{f_X:X\in \J_P\}$  satisfy properties (i) and (ii) in the definition of $\rC^*_s(P)$. Next, let $p_1,q_1,...,p_n,q_n \in P$ be elements satisfying $p_1^{-1}q_1...p_n^{-1}q_n = e$ in $G$. Put $Y=q_n^{-1}p_n ... q_1^{-1}p_q P$. We claim that 
\[
w_{p_1}^*w_{q_1} ... w_{p_n}^*w_{q_n}=f_Y.
\]
By \eqref{Eq:missingpsi}, we get that
$$
\psi(Y) = \psi(q_n)^{-1}\psi(p_n) ... \psi(q_1)^{-1}\psi(p_1)\psi(P).
$$
On the other hand, using that $\psi(e)=e$ and $\psi(P)\subset Q$ we see that $\psi(P)Q=Q$. Therefore,
$$
\psi(Y)Q = \psi(q_n)^{-1}\psi(p_n) ... \psi(q_1)^{-1}\psi(p_1)Q.
$$
Observe now that the collections $\{\lambda_q^Q \in B(\ell^2(Q)):q\in Q\}$ and $\{1^Q_X:X\in \J_Q\}$ satisfy (iii), whence
\begin{align*}
w_{p_1}^*w_{q_1} \ldots w_{p_n}^*w_{q_n} &= \lambda^{Q*}_{\psi(p_1)} \lambda^{Q}_{\psi(q_1)} \ldots  \lambda^{Q*}_{\psi(p_n)} \lambda^{Q}_{\psi(q_n)}\\
&=  1^Q_{\psi(q_n)^{-1}\psi(p_n) ... \psi(q_1)^{-1}\psi(p_1)Q} \\
&= 1^Q_{\psi(Y)Q}= f_Y.
\end{align*}
This establishes the claim, and we conclude that the collection $\{w_p:p\in P\}$ and $\{f_X:X\in \J_P\}$ also satisfy (iii) in the definition of $\rC^*_s(P)$.
Consequently, there is a $*$-homomorphism $\Psi : \rC^*_s(P) \rightarrow \rC^*_r(Q)$ such that
\[
\Psi(v_p)=w_p, \quad p\in P
\]
and
\[
\Psi(e_X)=f_X, \quad X\in \J_P.
\]
Furthermore, given $p_1,\ldots,p_m,q_1,\ldots,q_m\in P$, we compute using \cite[Lemmas 3.1 and 3.13]{li2012} that
\begin{align*}
&E_r^Q \circ \Psi( v_{p_1}^* v_{q_1}\cdots v_{p_m}^* v_{q_m})=E_r^Q( \lambda^{Q*}_{\psi(p_1)} \lambda^Q_{\psi(q_1)}\cdots \lambda^{Q*}_{\psi(p_m)} \lambda^Q_{\psi(q_m)})\\
&=\begin{cases}
1^Q_{\psi(q_m^{-1})\psi(p_m)\cdots \psi(q_1^{-1})\psi(p_1)Q} & \text{ if } \psi( p_1^{-1}q_1\cdots p_{m}^{-1}q_m)=e \text{ in } H,\\
0 & \text{ otherwise}
\end{cases}\\
&=\begin{cases}
1^Q_{\psi(q_m^{-1})\psi(p_m)\cdots \psi(q_1^{-1})\psi(p_1)Q} & \text{ if } p_1^{-1}q_1\cdots p_{m}^{-1}q_m=e \text{ in } G,\\
0 & \text{ otherwise}
\end{cases}
\end{align*}
where the last equality follows since $\psi$ is injective. 
Using that $\psi(P)Q=Q$ again, \cite[Lemma 3.13]{li2012} implies that 
\[
E_r^Q \circ \Psi( v_{p_1}^* v_{q_1}\cdots v_{p_m}^* v_{q_m})=\Psi\circ E_s^P (v_{p_1}^* v_{q_1}\cdots v_{p_m}^* v_{q_m}).
\]
We conclude from this that 
\[
E_r^Q\circ \Psi=\Psi\circ E_s^P.
\]
Combining this equality with \eqref{Eq:CE-kern} and recalling that $E_r^Q$ is faithful, we infer that $\ker \lambda_*^P\subset \ker \Psi$. Since $\rC^*_r(P)\cong \rC^*_s(P)/\ker \lambda_*^P$, we obtain a $*$-representation $\Psi_r:\rC^*_r(P)\to \rC^*_r(Q)$ with the desired property.
\end{proof}

Regarding \eqref{Eq:missingpsi}, we remark that it is always true, for $p\in P$ and $X\subset P$, that
\begin{equation}\label{Eq:psi}
\psi(p^{-1}X)=(\psi(p)^{-1}\psi(X))\cap \psi(P)\subset \psi(p)^{-1}\psi(X)
\end{equation}
if $\psi$ is an injective homomorphism with $\psi(P)\subset Q$.  In a previous version of the paper the assumption that \eqref{Eq:missingpsi} holds in Proposition \ref{p:injectivity} was missing. Without this assumption, a counterexample to Proposition \ref{p:injectivity} can be produced as follows. We thank an anonymous referee bringing this issue to our attention.

\begin{example}\label{E:issue}
Let $P=\mathbb{F}_2^+$ be the free monoid on two generators $a$ and $b$, which is independent. Let $G= H = Q = \mathbb{F}_2$ be the free group on $a$ and $b$. Let $\psi : \mathbb{F}_2 \rightarrow \mathbb{F}_2$ be the identity map. Clearly, $\psi$ is an injective homomorphism with $\psi(P)\subset Q$. Further, we find that $a^{-1} \in \psi(a)^{-1} \psi(P) \setminus \psi(P)$, so by virtue of \eqref{Eq:psi} we see that \eqref{Eq:missingpsi} fails in this case. 
%

It is well-known that $C^*_r(\mathbb{F}_2^+)$ is the Cuntz--Toeplitz algebra, which is the universal C*-algebra generated by isometries $\lambda^P_a$ and $\lambda^P_b$ with pairwise orthogonal ranges. On the other hand, $\lambda^Q_a$ and $\lambda^Q_b$ are unitaries. Thus, there is no $*$-homomorphism that maps $\lambda^P_a \mapsto \lambda^Q_a$ and $\lambda^P_b \mapsto \lambda^Q_b$. \qed
\end{example}

We now give an extension \cite[Lemma 7.5]{CEL2015} that does not require the semigroups to be Ore. Furthermore, it works beyond the ``diagonal" action $\lambda^P_p \mapsto \lambda^P_p \otimes \lambda^P_p$.

\begin{proposition} \label{P:tensor}
Let $P$ and $Q$ be semigroups embeddable into some groups $G$ and $H$ respectively. Assume that $P$ is independent and let $\varphi : G \rightarrow H$ be a homomorphism such that $\varphi(P) \subset Q$. Then there is a $*$-representation $\Delta_{\varphi} : \rC^*_r(P) \rightarrow \rC^*_r(P) \otimes_{\min} \rC^*_r(Q)$ such that
\[
\Delta_{\varphi}(\lambda^P_p) = \lambda^P_p \otimes \lambda^Q_{\varphi(p)}
\]
for every $p\in P$.
\end{proposition}
\begin{proof}
Let $\psi : G \rightarrow G \times H$ be the injective homomorphism given by $\psi(g) = (g,\varphi(g))$. Clearly $\psi(P) \subset P \times Q$. We claim that $\psi$ satisfies \eqref{Eq:missingpsi}. To see this, fix $p\in P$ and $X\in \J_P$. If $(a,b)\in \psi(p)^{-1}\psi(X)$, then there is $x\in X$ such that 
\[
(pa,\phi(p)b)=\psi(p)(a,b)=\psi(x)=(x,\phi(x)).
\]
Equivalently, $pa=x$ and $\phi(p)b=\phi(x)$, which forces $\phi(p)b=\phi(p)\phi(a)$, or $\phi(a)=b$. We conclude that $(a,b)=\psi(a)\in \psi(P)$. Hence $\psi(p)^{-1}\psi(X)\subset \psi(P)$, so \eqref{Eq:psi} implies that \eqref{Eq:missingpsi} holds.
Hence, by Proposition \ref{p:injectivity} we have a $*$-homomorphism $\Psi_r : \rC^*_r(P) \rightarrow \rC^*_r(P\times Q)$ such that 
\[
\Psi(\lambda^P_p) = \lambda^{P\times Q}_{(p,\varphi(p))}, \quad p\in P.
\] 
On the other hand, by \cite[Lemma 2.16]{li2012} there is a $*$-isomorphism $\Theta:\rC^*_r(P\times Q) \to \rC^*_r(P) \otimes_{\min} \rC^*_r(Q)$ such that
\[
\Theta(\lambda^{P\times Q}_{p,q})= \lambda^P_p \otimes \lambda^Q_q, \quad p\in P, q\in Q.
\]
The map  $\Delta_\phi=\Theta\circ \Psi_r$ has the required properties.
\end{proof}

Our next goal is to show that the map $\Delta_\phi$ above restricts to a coaction of $Q$ on $\A_r(P)$. To do this, we need the following simple observation.

\begin{lemma}[Fell's absorption principle for semigroups] \label{l:sfap}
Let $P$ be a semigroup, and suppose $V: P \rightarrow B(\fH)$ is a representation of $P$ by isometries. Then, there is an isometry $W\in B(\ell^2(P)\otimes \fH)$ such that
\[
W (\lambda_p\otimes \id_{\fH})=(\lambda_p\otimes V_p)W, \quad p\in P.
\]
\end{lemma}
\begin{proof}
The map
\[
e_p \otimes \xi\mapsto e_p \otimes V_p(\xi), \quad p\in P, \xi\in \fH
\]
is easily seen to extend to an isometry  $W:\ell^2(P)\otimes \fH\to \ell^2(P)\otimes \fH$.  But then, for $p,q \in P$ and $\xi \in \fH$ we have
\begin{align*}
(\lambda_q \otimes V_q) W (e_p \otimes \xi) &= e_{qp} \otimes V_{qp}(\xi) = W(e_{qp} \otimes \xi) \\
&= W(\lambda_q \otimes \id_{\fH})(e_p \otimes \xi).
\end{align*}
We infer that $W (\lambda_q\otimes \id_{\fH})=(\lambda_q\otimes V_q)W$ as required.
\end{proof}

We now arrive at a result which will provide us with many examples of coactions.

\begin{theorem} \label{t:coaction}
Let $P$ and $Q$ be semigroups embeddable into groups $G$ and $H$ respectively. Assume that $P$ is independent and let $\varphi : G \rightarrow H$ be a homomorphism such that $\varphi(P) \subset Q$. Then, there is a coaction $\delta$ of $Q$ on $\A_r(P)$ such that
\[
\delta(\lambda^P_p)=\lambda^P_p\otimes \lambda^Q_{\phi(p)}, \quad p\in P.
\]
\end{theorem}

\begin{proof}
By virtue of Proposition \ref{P:tensor}, there is a $*$-representation $\Delta_{\varphi} : \rC^*_r(P) \rightarrow \rC^*_r(P) \otimes_{\min} \rC^*_r(Q)$ such that
\[
\Delta_{\varphi}(\lambda^P_p) = \lambda^P_p \otimes \lambda^Q_{\varphi(p)}
\]
Note that the injectivity of the minimal tensor product we have $\A_r(P) \otimes_{\min} \A_r(Q) \subset \rC_r^*(P) \otimes_{\min} \rC_r^*(Q)$. Thus, we obtain a representation $\delta:\A_r(P)\to \A_r(P)\otimes_{\min}\A_r(Q)$ by putting $\delta=\Delta_\phi|_{\A_r(P)}$. In particular, we conclude that $\delta$ takes values in $\A_r(P)\otimes_{\min}\A_r(Q)$, as required.   By construction, it is clear that the sum of the spectral subspaces for $\delta$ is norm dense in $\A_r(P)$.

Let $V : P \rightarrow B(\ell^2(Q))$ be given by $V_p = \lambda^Q_{\varphi(p)}$, which is a representation of $P$ by isometries. By Lemma \ref{l:sfap} there is an isometry $W$ such that 
\[
W (\lambda^P_p\otimes \id_{\fH})=\delta(\lambda^P_p)W, \quad p\in P.
\]
Hence, we have
\[
W(a\otimes \id_\fH)=\delta(a)W, \quad a\in \A_r(P).
\]
Let $A=[a_{ij}]\in \bM_n(\A_r(P))$ for some $n\in \bN$. We find
\begin{align*}
\|A\|&=\|[a_{ij}\otimes I_\fH]\|=\| [Wa_{ij}\otimes I_\fH]\|\\
&=\| \delta^{(n)}(A)(W\oplus W\oplus \ldots \oplus W)\|\leq \|\delta^{(n)}(A)\|.
\end{align*}
We infer that $\delta$ is completely isometric, and thus is indeed a coaction of $Q$ on $\A_r(P)$.
 \end{proof}

The following shows that the notion of a semigroup coaction in Definition \ref{d:coaction} is deserving of its name, at least for independent semigroups embedded in groups.

\begin{corollary}
Let $P$ be an independent semigroup embedded in a group $G$. Then there is a comultiplication on $\A_r(P)$, which is the completely isometric homomorphism $\Delta_P : \A_r(P) \rightarrow \A_r(P) \otimes_{\min} \A_r(P)$ given by $\Delta_P(\lambda_p^P) = \lambda_p^P \otimes \lambda_p^P$.

If moreover $P$ acts on some operator algebra $\A$ by $\delta$, then we automatically have the coaction identity $(\delta \otimes \id) \circ \delta = (\id \otimes \Delta_P) \circ \delta$.
\end{corollary}

\begin{proof}
The first part follows directly from Theorem \ref{t:coaction} by taking $\varphi$ to be the identity on $G$. For the second part, by definition of coaction it suffices to verify the coaction identity on each $\A_p^{\delta}$, but this is easily deduced.
\end{proof}

The last piece of the puzzle is to verify that the coaction constructed in Theorem \ref{t:coaction} is in fact RFD. For this purpose, the following elementary fact will be useful.

\begin{lemma}\label{L:spectralsub}
Let $P$ and $Q$ be semigroups. Let $\phi:P\to Q$ be a homomorphism and let $\delta:\A_r(P)\to \A_r(P)\otimes_{\min} \A_r(Q)$ be a representation such that
\[
\delta(\lambda^P_p)=\lambda^P_p\otimes \lambda^Q_{\phi(p)}, \quad p\in P.
\]
Then, for each $q\in Q$ the set $\{a\in \A_r(P):\delta(a)=a\otimes \lambda^Q_q\}$ coincides with the norm closure of $\sum_{p\in \phi^{-1}(q)}\bC \lambda^P_p$.
\end{lemma}
\begin{proof}
Fix $q\in Q$ and $a\in \A_r(P)$ such that $\delta(a)=a\otimes \lambda^Q_q$. Let $\eps>0$. Then there is a finite linear combination $a'=\sum_{p\in P}c_p \lambda^P_p$ with $c_p\in \bC$ such that $\|a-a'\|<\eps$. Let $h\in \ell^2(P)$ be a unit vector and let $s\in Q$. Note that
\[
\delta(a')(h\otimes e_s)=\sum_{p\in P} c_p \lambda^P_p h\otimes e_{\phi(p)s} \ \ \text{and} \ \ \delta(a)(h\otimes e_s)=ah\otimes e_{qs}.
\]
Since $Q$ is cancellative, we see that the set $\{e_{rs}:r\in Q\}$ is orthonormal, whence by applying Pythagoras' theorem we get
\begin{align*}
&\left\| \left(\sum_{p\in \phi^{-1}(q)} c_p \lambda^P_p-a\right)h\right\|^2 \\
&\leq \left\| \sum_{p\notin \phi^{-1}(q)} c_p \lambda^P_p h\otimes e_{\phi(p)s} \right\|^2 + \left\|\left(\sum_{p\in \phi^{-1}(q)} c_p \lambda^P_p-a\right )h\otimes e_{qs}\right\|^2\\
&=\|\delta(a'-a)(h\otimes e_s)\|^2\leq \eps^2.
\end{align*}
Since $\eps>0$ and the unit vector $h\in  \ell^2(P)$ were chosen arbitrarily, this implies that $a$  lies in the norm closure of $\sum_{p\in \phi^{-1}(q)}\bC \lambda^P_p$. The reverse inclusion is obvious.
\end{proof}

Let $P$ and $Q$ be semigroups embedded in some groups $G$ and $H$ respectively. A homomorphism $\varphi : G \rightarrow H$ is a \emph{$(P,Q)$-map} if $\varphi(P) \subset Q$ and the set $\phi^{-1}(q)\cap P$ is finite for every $q\in Q$. It is readily verified that in the presence of a $(P,Q)$-map with $\phi(P)=Q$, if $P$ has FDP then so does $Q$. 
Surjective $(P,Q)$-maps with $\phi(P)=Q$ are sometimes called controlled maps. These were originally introduced in \cite{LR1996} as generalizations of length functions. Various kinds of controlled maps were used successfully in the literature to establish amenability and nuclearity properties for semigroup $\rC^*$-algebras \cite{CL2002, CL2007, aHRT2018, aHNSY2021}.

We now arrive at one of the main results of the paper. Recall that a cancellative semigroup $P$ is said to be \emph{left-amenable} if it admits a left-invariant mean. More precisely, there is a state $\mu$ on $\ell^{\infty}(P)$ such that for every $p\in P$, and $f\in\ell^{\infty}(P)$ we have that $\mu(p\cdot f) =\mu(f)$ where $(p \cdot f)(q) = f(pq)$. For instance, all abelian semigroups are left-amenable.

\begin{theorem}\label{T:main-semigroup}
Let $P$ and $Q$ be countable semigroups embedded in some groups $G$ and $H$ respectively. Assume that $P$ and $Q$ are independent, and that $Q$ is left-amenable and has FDP. Assume also that there exists a $(P,Q)$-map $\varphi :G \rightarrow H$. Then, there is an RFD coaction of $Q$ on $\A_r(P)$ and $\rC^*_{\max}(\A_r(P))$ is RFD.
\end{theorem}

\begin{proof}
By Theorem \ref{t:coaction}, there is a coaction $\delta$ of $Q$ on $\A_r(P)$ such that
\[
\delta(\lambda^P_p) = \lambda^P_p \otimes \lambda^Q_{\varphi(p)}, \quad p\in P.
\] 
For $q\in Q$ we let
\[
\B_q=\{a\in \A_r(P):\delta(a)=a\otimes \lambda^Q_q \}
\]
be the corresponding spectral subspace for $\delta$. For each $q\in Q$, we see by Lemma \ref{L:spectralsub} that $\B_q$ is the norm closure of $\sum_{p\in \phi^{-1}(q)}\bC \lambda^P_p$. 

Let $(\pi_F)$ be the Exel--Loring approximation for the identity representation of $\A_r(Q)$ given by Proposition \ref{P:ArP}. Thus, for each finite subset $F\subset Q$ we see that
\[
Q \setminus \left( \cup_{q\in F}\cup_{r\in \R_q}\L_r\right)=  (\lambda^Q)^{-1}(\ker \pi_F).
\] 
By definition of a coaction, we know that $\A_r(P)$ is the norm closure of 
\[
\sum_{q\in Q}\B_q=\sum_{q\in Q}\sum_{p\in \phi^{-1}(q)}\bC \lambda^P_p.
\]
We conclude that the quotient $\Q_F$ of $\A_r(P)$ by the norm closure of the ideal generated by
$
\sum_{q \in (\lambda^Q)^{-1}(\ker \pi_F)}\B_q
$
is spanned by the image in the quotient of
\[
\sum_{q\in F}\sum_{r\in \R_q}\sum_{p\in \varphi^{-1}(\L_r)}\bC \lambda_p^P.
\]
Since $Q$ has FDP and $\varphi$ is a $(P,Q)$-map, we get that $\Q_F$ is finite-dimensional. We infer that $\rC^*_{\max}(\Q_F)$ is RFD by virtue of \cite[Theorem 5.1]{CR2018rfd}. Thus, we see that $\delta$ is an RFD coaction.

Finally, since $Q$ is left-amenable and independent, by \cite[Theorem 6.42]{li2017} there is a character $\chi:\rC^*_r(Q)\to \bC$. Clearly, for every $q\in Q$ we have $\chi(\lambda^Q_q)=1$  since $\lambda^Q_q$ is an isometry.  Thus, by Theorem \ref{T:coactionC*maxRFD} we get that $\rC^*_{\max}(\A_r(P))$ is RFD.
\end{proof}

We now extract some applications of the previous result.

\begin{corollary}
Let $P$ be a countable, independent semigroup with FDP that is embedded in an abelian group. Then $\rC^*_{\max}(\A_r(P))$ is RFD. 
\end{corollary}
\begin{proof}
Every abelian semigroup is left-amenable. Hence, in Theorem \ref{T:main-semigroup} we may take $\varphi$ to be the identity, so that $\rC^*_{\max}(\A_r(P))$ is RFD.
\end{proof}

In paricular, the previous result shows that $\rC^*_{\max}(\A_r(\bN^d))$ is RFD for every $d \in \bN$ (including $d=\aleph_0$). We now study some classical examples of semigroups.

\begin{example}[Left-angled Artin monoids] \label{ex:laam}
Let $\Gamma = (V,E)$ be a finite undirected graph, where two vertices are connected by at most one edge and no vertex is connected to itself. We let $A_{\Gamma}$ be the group generated by $V$ as formal generators where, for $v,w\in W$, we have $vw=wv$ if  if there is an edge between $v$ in $w$.
We then define the unital semigroup $A_{\Gamma}^+$ as the semigroup generated by the identity and $\{v:v\in V\}$. It is known that right-angled Artin groups are quasi-lattice ordered by their monoids (see \cite{CL2002}), so we get that $A_{\Gamma}^+$ is right LCM, and is hence independent.

Next, consider the group $\bZ^V$. For each $v\in V$, let $e_v\in \bZ^V$ denote the indicator function of $\{v\}$.  There is a surjective group homomorphism $\varphi : A_{\Gamma} \rightarrow \mathbb{Z}^V$ such that $\varphi(v) = e_v$. Viewing $\mathbb{N}^V$ as a subsemigroup of $\mathbb{Z}^V$, it is clear that  $\varphi(A_{\Gamma}^+) = \mathbb{N}^V$ and $\phi^{-1}(f)$ is finite for every $f\in \bN^V$. Thus, $\varphi$ is a $(A_\Gamma^+,\bN^V)$-map. As a consequence of Theorem \ref{T:main-semigroup}, we get that $\rC^*_{\max}(\A_r(A^+_{\Gamma}))$ is RFD.
\qed
\end{example}

\begin{example}[Braid monoids] \label{ex:braid}
For $n\geq 3$ let $B_n$ be the group generated by elements $\sigma_1,...,\sigma_{n-1}$ subject to the relations
$$
\sigma_i\sigma_{i+1}\sigma_i = \sigma_{i+1}\sigma_i\sigma_{i+1} \ \text{for} \ 1\leq i \leq n-2
$$
and
$$
\sigma_i\sigma_j = \sigma_j\sigma_i \ \text{for} \ |i-j|\geq 2.
$$
We let $B_n^+$ be the unital semigroup generated by $\sigma_1,...,\sigma_{n-1}$ in $B_n$.

For each $n\geq 3$, it follows from \cite[Lemma 2.1]{BS1972} that $B_n^+$ is right LCM, and hence is independent (see also the discussion preceding \cite[Lemma 6.33]{li2017}). There is a surjective group homomorphism $\varphi : B_n \rightarrow \mathbb{Z}$ such that $\varphi(\sigma_i) = 1$ for every $i$.  It is readily seen that $\phi$ is a $(B_n^+,\bN)$-map. By Theorem \ref{T:main-semigroup}, we get that $\rC^*_{\max}(\A_r(B_n^+))$ is RFD.
\qed
\end{example}

%
%

\subsection{Operator algebras of functions}\label{SS:OAfunctions}
Let $\Omega\subset \bC^n$ be a subset and let $\bT\subset \bC$ denote the unit circle. Assume that $\Omega$ is \emph{balanced}, so that $\zeta z\in \Omega$ for every $z\in \Omega$ and every $\zeta\in \bT$. In this case, given a function $f:\Omega\to \bC$ and $\zeta\in \bT$, we define $f_\zeta:\Omega\to \bC$ to be the function
\[
f_\zeta(z)=f(\zeta z), \quad z\in \Omega.
\]
Let $\A$ be an operator algebra consisting of functions on $\Omega$.  We say that $\A$ is \emph{$\bT$-symmetric}  if for each $\zeta\in \bT$, the map 
\[
f\mapsto f_\zeta, \quad f\in \A
\]
defines a completely isometric isomorphism $\Phi_\zeta:\A\to \A$. It would certainly be natural here to impose some form of continuity on the map  $\zeta\mapsto \Phi_\zeta$. Fortunately, in our cases of interest this turns out to be automatic.

\begin{lemma}\label{L:autcont}
Let $\Omega\subset\bC^n$ be a balanced subset and let $\A$ be a $\bT$-symmetric operator algebra of functions on $\Omega$. Assume that there is a collection of homogeneous polynomials spanning a dense subset of $\A$. Then, the map
\[
\zeta\mapsto \Phi_\zeta(f), \quad \zeta\in \bT
\]
is norm continuous for every $f\in \A$.
\end{lemma}
\begin{proof}
Given $f\in \A$, we define a function $\delta_0(f):\bT\to \A$ as
\[
\delta_0(f)(\zeta)= \Phi_\zeta(f), \quad \zeta\in \bT.
\]
For a homogeneous polynomial $f\in \A$ of degree $n$, we see that
\[
\delta_0(f)(\zeta)=\zeta^n f, \quad \zeta\in \bT
\]
so in particular $\delta_0(f)$ is continuous.  A standard $\frac{\eps}{3}$-argument using our density assumption shows that in fact $\delta_0(a)$ is continuous for every $a\in \A$. 
\end{proof}

Next, we show that $\bT$--symmetric operator algebras admit RFD coactions by $\bN$. 

\begin{theorem}\label{T:coactionAsymm}
Let $\Omega\subset\bC^n$ be a balanced subset and let $\A$ be a $\bT$-symmetric operator algebra of functions on $\Omega$. Assume that there is a collection of homogeneous polynomials spanning a dense subset of $\A$. Then, $\A$ admits an RFD coaction by $\bN$. In particular, $\rC^*_{\max}(\A)$ is RFD.
\end{theorem}
\begin{proof}
Given $f\in \A$, we define a function $\delta_0(f):\bT\to \A$ as
\[
\delta_0(f)(\zeta)= \Phi_\zeta(f), \quad \zeta\in \bT.
\]
By Lemma \ref{L:autcont}, we see that $\delta_0(f)\in \rC(\bT;\A)$ for every $f\in \A$. Hence, we obtain a completely isometric homomorphism $\delta_0:\A\to \rC(\bT; \A)$. 
Next,  it is a consequence of \cite[Proposition 12.5]{paulsen2002} and of the injectivity of the minimal tensor product that there is a completely isometric isomorphism $\Theta: \A\otimes_{\min} \rC(\bT) \to \rC(\bT; \A)$ such that 
\[
\Theta(a\otimes g)(\zeta)=g(\zeta)a, \quad \zeta\in \bT
\]
for every $a\in \A$ and $g\in \rC(\bT)$. If $f\in \A$ is a homogeneous polynomial of degree $n$, then we see that
\[
\Theta^{-1}(\delta_0(f))=f\otimes \zeta^n.
\]
We conclude that $(\Theta^{-1}\circ\delta_0)(\A)\subset \A\otimes_{\min}\rA(\bD)$, where $\rA(\bD)$ denotes the disc algebra.  

Recall now that there is a surjective $*$-representation $\tau:\rC^*_r(\bN)\to \rC(\bT)$ such that
\[
\tau( \lambda^\bN_n)=\zeta^n, \quad n\in \bN
\]
and whose restriction to $\A_r(\bN)$ is completely isometric. In particular, this clearly implies that there is a character $\chi:\rC^*_r(\bN)\to \bC$ with $\chi(\lambda^\bN_n)=1$ for every $n\in \bN$. Putting $\rho=(\tau|_{\A_r(\bN)})^{-1}$, we obtain a completely isometric isomorphism $\rho:\rA(\bD)\to \A_r(\bN)$ such that
\[
\rho(z^n)=\lambda^{\bN}_n, \quad n\in \bN.
\]
Define 
\[
\delta=(\id\otimes \rho)\circ \Theta^{-1}\circ \delta_0:\A\to \A\otimes_{\min} \A_r(\bN),
\]
which is a completely isometric representation. We claim that $\delta$ is an RFD coaction of $\bN$ on $\A$. 

For $n\in \bN$, we let 
\[
\A_n =\{a\in \A: \delta(a)=a\otimes \lambda^\bN_n \}
\]
be the corresponding spectral subspace. By construction, we see that  $a\in \A_n$ if and only if 
\[
\Phi_\zeta(a)=\delta_0(a)(\zeta)=\zeta^n a, \quad \zeta\in \bT.
\]
 Hence, $\A_n$ is the space of all homogeneous polynomials of degree $n$ in $\A$, and in particular is finite-dimensional. By assumption, we must have that $\sum_{n\in \bN} \A_n$ is norm dense in $\A$. This shows that $\delta$ is a coaction of $\bN$ on $\A$. It remains to show that this coaction is in fact RFD.
 
 Let $(\pi_F)$ be the Exel--Loring approximation for the identity  representation of $\A_r(P)$ from Proposition \ref{P:ArP}.
Given a finite subset $F\subset \bN$ and $n\in \bN$, we see from  Proposition \ref{P:ArP} that $\lambda^N_n\in \ker \pi_F$ if $n> \max F$.  Thus, the quotient $\Q_F$ of $\A_r(\bN)$ by the closed ideal generated by
\[
\sum_{n\in \lambda^{-1}(\ker \pi_F)} \A_n
\]
has dimension at most that of 
\[
\sum_{n\leq  \max F} \A_n
\]
and in particular is finite-dimensional. Hence, $\rC^*_{\max}(\Q_F)$ is RFD by \cite[Theorem 5.1]{CR2018rfd}. We conclude that  $\delta$ is indeed an RFD coaction of $\bN$ on $\A$. 
Finally,  invoking Theorem \ref{T:coactionC*maxRFD} we get that $\rC^*_{\max}(\A)$ is RFD.
\end{proof}

Many concrete examples of operator algebras can now be analyzed in light of the previous result.

\begin{corollary}\label{C:AOmega}
Let $\Omega\subset \bC^n$ be a balanced bounded subset and let $\rA(\Omega)$ denote the closure of the polynomials inside the continuous functions $\rC(\ol{\Omega})$. Then, $\rC^*_{\max}(\rA(\Omega))$ is RFD.
\end{corollary}
\begin{proof}
In view of Theorem \ref{T:coactionAsymm}, it suffices to show that $\rA(\Omega)$ is a $\bT$-symmetric algebra. For $f\in \rA(\Omega)$ and $\zeta\in \bT$, we see that
\[
\|f_\zeta\|=\max_{z\in\ol{\Omega}}|f(\zeta z)|=\max_{w\in \ol{\Omega}}|f(w)|=\|f\|, \quad f\in \rA(\Omega).
\]
This implies that the map $f\mapsto f_\zeta$ defines an isometric isomorphism $\Phi_\zeta:\rA(\Omega)\to \rA(\Omega)$.
The fact that $\Phi_\zeta$ is in fact completely isometric follows then from \cite[Theorem 3.9]{paulsen2002}.
\end{proof}

By taking $\Omega$ above to be the open unit polydisc or the  open unit ball, we conclude from Corollary \ref{C:AOmega} that the maximal $\rC^*$-covers of the polydisc algebra and of the ball algebra are necessarily RFD. This answers a question left open in \cite{CR2018rfd}. 

By definition, the algebra $\rA(\Omega)$ considered in Corollary \ref{C:AOmega} is a \emph{uniform algebra}, as it embeds into the commutative $\rC^*$-algebra $\rC(\ol{\Omega})$. Below, we deal with operator algebras of functions on $\Omega$ which do not necessarily admit such an embedding. 

Let $n$ be a positive integer. Let $\Gamma$ denote the set of (unitary equivalence classes of) commuting $n$-tuples of contractions on some separable Hilbert space. Let $p\in \bC[z_1,\ldots,z_n]$ be a polynomial in $n$ variables. The quantity 
\[
\sup_{(T_1,\ldots,T_n)\in \Gamma}\|p(T_1,\ldots,T_n)\|
\]
is easily seen to be finite and to define a norm on $\bC[z_1,\ldots,z_n]$. The corresponding completion of $\bC[z_1,\ldots,z_n]$ is an operator algebra \cite[Chapter 18]{paulsen2002}, usually referred to as the \emph{Schur--Agler class} \cite{agler1990}. We denote it by $\S_n$. It follows from the classical von Neumann and Ando inequalities that $\S_1\cong \rA(\bD)$ and $\S_2\cong \rA(\bD^2)$. However, for bigger $n$ the Schur--Agler class differs from the polydisc algebra. Nevertheless, we can still prove the following.

\begin{corollary}\label{C:SAclass}
Let $n$ be a positive integer and let $\S_n$ denote the corresponding Schur--Agler class. Then, $\rC^*_{\max}(\S_n)$ is RFD.
\end{corollary}
\begin{proof}
By definition we see that
\[
|p(z)|\leq \|p\|_{\S_n}, \quad z\in \bD^n
\]
for every polynomial $p$ in $n$ variables, so that there is a contractive inclusion $\S_n\hookrightarrow \rA(\bD^n)$. Next, we note that if $(T_1,\ldots,T_n)$ is a commuting $n$-tuple of contractions, then so is $(\zeta T_1,\zeta T_2,\ldots,\zeta T_n)$ for each $\zeta\in \bT$. Therefore, there is a completely isometric isomorphism $\Phi_\zeta:\S_n\to \S_n$ such that
\[
\Phi_\zeta (f)=f_\zeta, \quad f\in \S_n.
\]
Hence, we conclude that $\S_n$ is a $\bT$--symmetric operator algebra of functions on $\bD^n$, so that $\rC^*_{\max}(\S_n)$ is RFD by Theorem \ref{T:coactionAsymm}.
\end{proof}

For our next application, we will make use of standard terminology from the theory of reproducing kernel Hilbert spaces; details can be found in \cite{agler2002} and \cite{hartz2017isom}.

Let $\Omega\subset \bC^n$ be a balanced open connected subset containing the origin. Let $K:\Omega\times\Omega\to \bC$ be a kernel function normalized at the origin, so that $K(z,0)=1$ for every $z\in \Omega$. We say that $K$ is \emph{$\bT$-invariant} if  we have that
\begin{enumerate}
\item $K$ is holomorphic in the first variable,
\item for every $z,w\in \Omega$ the map
\[
\zeta \mapsto K(\zeta z,w), \quad \zeta\in \bT
\] is continuous,

\item for every $z,w\in \Omega$ and $\zeta\in \bT$ we have that $K(z,w) =K(\zeta z,\zeta w)$.
\end{enumerate}
Many important classical kernels satisfy these properties, such as the Drury--Arveson kernel and the Dirichlet kernel; see \cite[Example 6.1]{hartz2017isom}.

Let $\fH$ be the reproducing kernel Hilbert space corresponding to $K$ and let $\Mult(\fH)$ denote its multiplier algebra. By our assumption on $K$, we see that $\fH$ consists of holomorphic functions on $\Omega$. Furthermore, because $1\in \fH$ we infer $\Mult(\fH)\subset \fH$. For each $\zeta\in \bT$, the map $\Gamma_\zeta:\fH\to \fH$ defined as
\[
(\Gamma_\zeta(f))(z)=f(\zeta z), \quad f\in \fH, z\in \Omega
\]
is easily checked to be a strongly continuous unitary path (see \cite[Section 6]{hartz2017isom}). For each $\zeta\in \bT$ and $\phi\in \Mult(\fH)$, we note that
\[
\Gamma_\zeta^* M_\phi \Gamma_\zeta= M_{\Gamma_{\ol{\zeta}}(\phi)}.
\]
Hence, we obtain a completely isometric isomorphism $\Theta_\zeta:\Mult(\fH)\to \Mult(\fH)$ defined as
\[
\Theta_\zeta(M_\phi)=M_{\Gamma_\zeta(\phi)}, \quad \phi\in \Mult(\fH).
\]
For each $n\in \bN$, we let $\rA(\fH)_n\subset \Mult(\fH)$ be the subspace consisting of those multipliers $\phi$ such that $\Theta_\zeta(M_\phi)=\zeta^n M_\phi$ for every $\zeta\in \bT$.
As explained in \cite[Example 6.1 (i)]{hartz2017isom}, for $n\geq 0$ we get that $\rA(\fH)_n$ consists of homogeneous polynomials of degree $n$, while $\rA(\fH)_n=\{0\}$ for $n<0$.
We define $\rA(\fH)$ to be the norm closure of $\sum_{n=0}^\infty \rA(\fH)_n$ inside of $\Mult(\fH)$. 

We now arrive at another application of Theorem \ref{T:coactionAsymm}. 

\begin{corollary}\label{C:RKHS}
Let $\Omega\subset \bC^n$ be a balanced open connected set containing the origin. Let $K$ be a $\bT$-invariant kernel on $\Omega$ and let $\fH$ be the corresponding reproducing kernel Hilbert space. Then, $\rC^*_{\max}(\mathrm{A}(\fH))$ is RFD.
\end{corollary}
\begin{proof}
We make a preliminary observation.  Let $\phi\in \rA(\fH)$. By definition of $\rA(\fH)$, there is a sequence $(p_n)$ of polynomials converging to $\phi$ in the norm topology of $\Mult(\fH)$. Since the multiplier norm always dominates the supremum norm over $\Omega$, we conclude that the sequence $(p_n)$ converges to $\phi$ uniformly on $\Omega$, and thus $\phi$ is holomorphic on $\Omega$. Hence, $\rA(\fH)$ is an operator algebra of holomorphic functions on $\Omega$.

In view of Theorem \ref{T:coactionAsymm}, it suffices to show that $\rA(\fH)$ is a $\bT$-symmetric algebra. For each $\zeta\in \bT$, we let $\Phi_\zeta=\Theta_\zeta|_{\rA(\fH)}$. It is clear that $\Phi_\zeta( \rA(\fH)_n)\subset \rA(\fH)_n$ for every $n\geq 0$, so that $\Phi_\zeta( \rA(\fH))\subset \rA(\fH)$. Thus, $\Phi_\zeta:\rA(\fH)\to \rA(\fH)$ is a completely isometric isomorphism. 
\end{proof}

A special class of $\bT$-invariant kernels is that of regular unitarily invariant kernels on the ball $\bB_n$. In this case, the algebra $\rA(\fH)$ was studied in \cite{AHMR2020subh} from the point of view of residual finite-dimensionality. It was shown therein (see the proof of  \cite[Theorem 4.2]{AHMR2020subh}) that the stronger requirement that $\rA(\fH)$ be subhomogeneous imposes tremendous restrictions, and in fact implies that $\rA(\fH)$ can be identified with the ball algebra. This provides further impetus to revealing whether $\rA(\fH)$ enjoys other refinements of residual finite-dimensionality.  Corollary \ref{C:RKHS} provides the first instance in this direction.

%
%
%
%


\bibliography{biblio_fdimapprox}
\bibliographystyle{plain}

\end{document}